\documentclass[times]{amsart}

\usepackage{amsmath,amsthm,amssymb}
\usepackage{tikz,graphicx}
\usepackage{standalone}
\usepackage{caption}
\usepackage{enumitem}
\usepackage{tkz-euclide}
\usepackage{bookmark}

\usetikzlibrary{intersections,calc}

\definecolor{blue-violet}{rgb}{0.54, 0.17, 0.89}

\tikzstyle{PurpleLine}=[line width=0.3mm,color=blue-violet,text=black]
\tikzstyle{PurplePoly}=[PurpleLine,fill=blue-violet!30]
\tikzstyle{BlueLine}=[line width=0.3mm,color=blue,text=black]
\tikzstyle{BluePoly}=[BlueLine,fill=blue!20]
\tikzstyle{RedLine}=[line width=0.3mm,color=red,text=black]
\tikzstyle{RedPoly}=[RedLine,fill=red!20]
\tikzstyle{GreenLine}=[thick,draw=black!30!green,text=black]
\tikzstyle{GreenPoly}=[thick,draw=green!50!black,fill=green!30,join=bevel]
\tikzstyle{OrangeLine}=[thick,color=orange]
\tikzstyle{GrayLine}=[thick,color=black!50!gray]
\tikzstyle{GrayPoly}=[GrayLine,fill=gray!20]
\tikzstyle{dot}=[shape=circle,draw,color=black,fill=black,inner sep=1.5pt]
\tikzstyle{bigdot}=[dot,inner sep=2pt]
\tikzstyle{littledot}=[dot,inner sep=1.2pt]
\tikzstyle{disk}=[thick,shape=circle,draw,color=black,fill=yellow!10]
\tikzstyle{plate}=[thick,shape=rectangle,draw,color=black,fill=yellow!10,
	rounded corners,minimum size=1.1cm]

\tikzstyle{dot}=[shape=circle,draw,color=black,fill=black,inner sep=1.5pt]
\tikzstyle{opendot}=[dot,fill=white]

\tikzstyle{YellowRect} = [shape=rectangle,rounded corners,draw,fill=yellow!40,minimum width = 2cm, minimum height=1cm]

\theoremstyle{plain}
\newtheorem{thm}{Theorem}[section]
\newtheorem{lem}[thm]{Lemma}
\newtheorem{cor}[thm]{Corollary}
\newtheorem{prop}[thm]{Proposition}

\newtheorem{mainthm}{Theorem}

\theoremstyle{definition}
\newtheorem{defn}[thm]{Definition}
\newtheorem{rem}[thm]{Remark}

\newtheorem{example}[thm]{Example}

\newcommand{\C}{\mathbb{C}}

\newcommand{\NC}{\textsc{NC}}
\newcommand{\BS}{\textsc{BS}}
\newcommand{\conv}{\textsc{Conv}}
\newcommand{\bool}{\textsc{Bool}}

\newcommand{\interior}{\text{int}}

\newcommand{\aext}{\mathcal{A}^\text{ex}}

\begin{document}

\title[Noncrossing Partition Lattices from Planar Configurations]{Noncrossing Partition Lattices \\ from Planar Configurations}
\author[S. Cohen]{Stella Cohen}
\email{scohen2@swarthmore.edu}
\address{Dept. of Mathematics \& Statistics, Swarthmore College, 
Swarthmore, PA 19081}
\author[M. Dougherty]{Michael Dougherty}
\email{doughemj@lafayette.edu}
\address{Dept. of Mathematics, Lafayette College,
  Easton, PA 18042}
\author[A. Harsh]{Andrew D. Harsh}
\email{aharsh1@swarthmore.edu}
\address{Dept. of Mathematics \& Statistics, Swarthmore College, 
Swarthmore, PA 19081}
\author[S. Martin]{Spencer Park Martin}
\email{smartin4@swarthmore.edu}
\address{Dept. of Mathematics \& Statistics, Swarthmore College, 
Swarthmore, PA 19081}
\date{\today}

\begin{abstract}
    The lattice of noncrossing partitions is well-known for its 
    wide variety of combinatorial appearances and properties.
    For example, the lattice is rank-symmetric and enumerated 
    by the Catalan numbers. In this article, we introduce a large family of 
    new noncrossing partition lattices with both of these properties, 
    each parametrized by a configuration of $n$ points in the plane.
\end{abstract}



\maketitle

\section{Introduction}

For each subset $P$ of the complex plane, a \emph{noncrossing partition}
of $P$ is a way of dividing $P$ into subsets with pairwise disjoint
convex hulls. The collection of all noncrossing partitions of $P$,
denoted $\NC(P)$, is a partially ordered set under refinement. 
When $P$ is the vertex set for a convex $n$-gon, $\NC(P)$ is
the classical noncrossing partition lattice $\NC_n$ introduced
by Kreweras \cite{kreweras72}. Among other things, Kreweras showed that the
size of $\NC_n$ is counted by the combinatorially ubiquitous Catalan numbers
$C_n = \frac{1}{n+1}\binom{2n}{n}$ and, more specifically, the number of 
lattice elements with rank $k$ is the Narayana number $N_{n,k} = \frac{1}{n}
\binom{n}{k}\binom{n}{k-1}$. Since $N_{n,k} = N_{n,n-k}$, this further says
that $\NC_n$ is a \emph{rank-symmetric} lattice.
In the fifty years since its definition,
the noncrossing partition lattice has made countless appearances in 
algebraic and geometric combinatorics - see the survey articles 
\cite{mccammond06} and \cite{baumeister19} for more information.

Returning to the more general case prompts a natural question: 
for which subsets $P \subset \mathbb{C}$ does the poset $\NC(P)$ 
have similar properties to $\NC_n$? While some existing 
work studies the size of $\NC(P)$ in asymptotic and 
extremal cases (e.g. \cite{sharirwelzl06} \cite{razen13}), 
similarities to $\NC_n$ seem uncommon in the literature. In our first main theorem,
we introduce a convexity condition on $P$ which guarantees that $\NC(P)$ 
has the same size as $\NC_n$.

Let $P \subset \mathbb{C}$ be a set of $n$ points in general position.
We say that a subset $A\subseteq P$ is \emph{convex} if no point in $A$
lies in the convex hull of the others, and we say that $P$ satisfies
\emph{Property~$\Delta_k$} if, for all convex $A\subseteq P$, the convex
hull of $A$ contains at most $|A| + k - 3$ elements of $P$ in its interior. 

\begin{mainthm}[Theorem~\ref{thm:catalan-counting}]
    \label{mainthm:catalan-counting}
    Let $P \subset \C$ be a set of $n$ points with Property~$\Delta_1$. 
    Then $\NC(P)$ is a rank-symmetric graded 
    lattice, and the number of elements with rank $k$ is
    the Narayana number $N_{n,k}$. In particular, $|\NC(P)| = C_n = |\NC_n|$.
\end{mainthm}

\begin{figure}
	\centering
    \includegraphics[width=\textwidth]{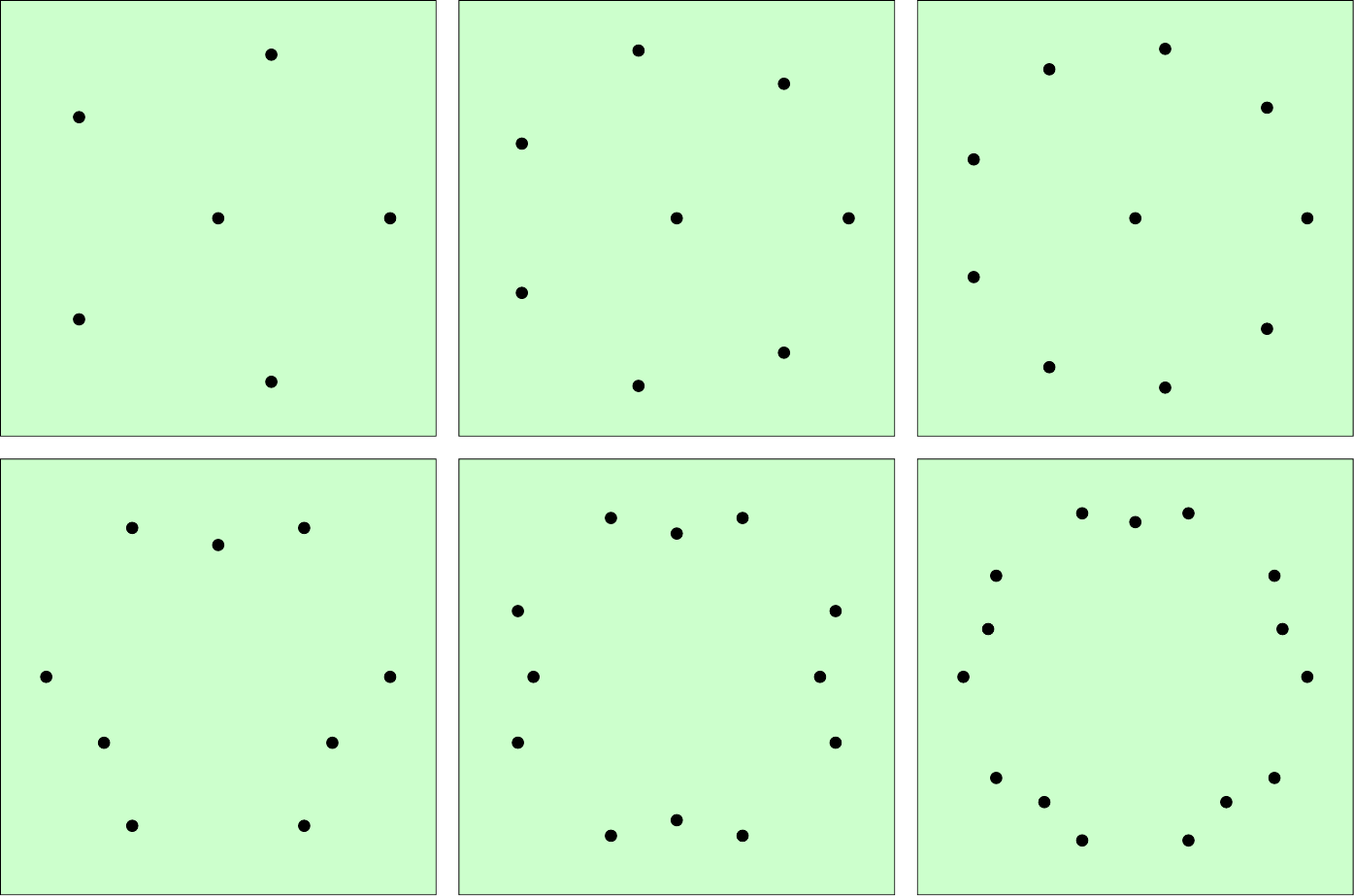}
    \caption{Some examples of configurations which satisfy 
    Property~$\Delta_1$}
    \label{fig:delta1-examples}
\end{figure}

It is worth noting that if $P$ contains a point which lies in the convex
hull of the others, then $\NC(P)$ is not isomorphic to $\NC_n$. Thus, 
Theorem~\ref{mainthm:catalan-counting} introduces a large new class of lattices
with the same number of elements in each rank as the noncrossing partition lattice.

If the conditions on $P$ are weakened to only require 
Property~$\Delta_2$, then the number of noncrossing partitions may increase
compared to those in Theorem~\ref{mainthm:catalan-counting}. Nevertheless, 
some symmetry is preserved.

\begin{mainthm}[Theorem~\ref{thm:rank-symmetry}]
    \label{mainthm:rank-symmetry}
    Let $P \subset \C$ be a set of $n$ points
    with Property~$\Delta_2$. Then $\NC(P)$ is a rank-symmetric graded 
    lattice.
\end{mainthm}

\begin{figure}
	\centering
    \includegraphics[width=\textwidth]{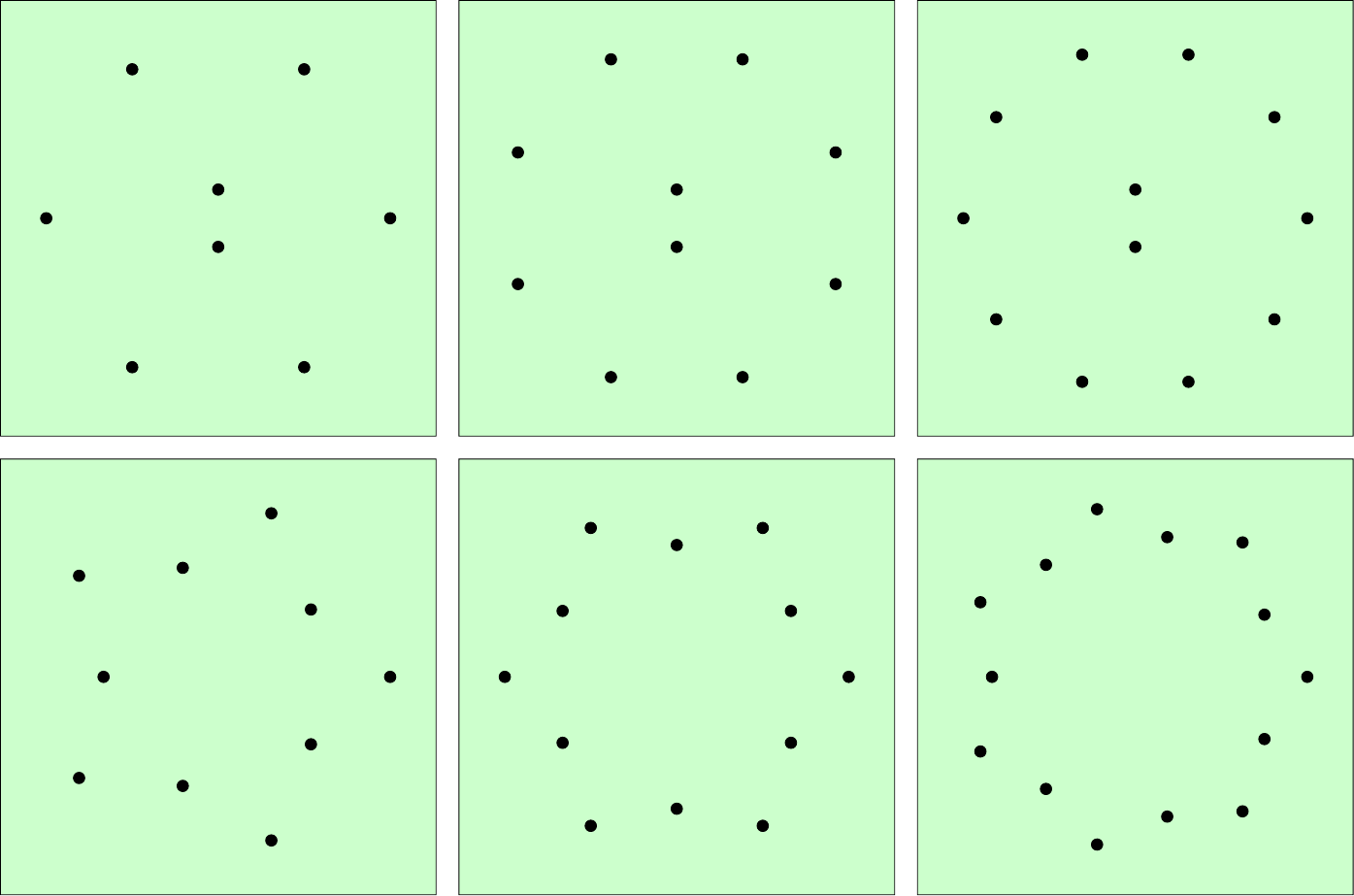}
    \caption{Some examples of configurations which satisfy 
    Property~$\Delta_2$}
    \label{fig:delta2-examples}
\end{figure}

See Figures~\ref{fig:delta1-examples} and \ref{fig:delta2-examples}
for examples of configurations which satisfy Properties~$\Delta_1$ and
$\Delta_2$ respectively.
One can quickly find examples of $P$ without Property~$\Delta_1$ such that
$\NC(P)$ fails to be rank-enumerated by the Narayana numbers and 
examples of $P$ without Property~$\Delta_2$ such that $\NC(P)$ fails to be
rank-symmetric. In fact, $\NC(P)$ can even fail to be graded if $P$ does not have Property~$\Delta_2$
(see Proposition~\ref{prop:graded} and Figure~\ref{fig:ungraded}). The extent to which these properties
are necessary conditions for Theorems~\ref{mainthm:catalan-counting}
and \ref{mainthm:rank-symmetry} is not immediately clear; this is an 
interesting direction for future research.

Some of the techniques used in proving Theorems~\ref{mainthm:catalan-counting} and
\ref{mainthm:rank-symmetry} can be interpreted in a stronger topological
context. Recall that the \emph{(unordered) configuration space of $n$ points in $\C$} 
is the topological space of all $n$-tuples in $\C^n$ with distinct entries, 
considered up to permutations of the coordinates. Also, if $P$ barely fails to be in 
general position (i.e. there is a single triple of collinear points in $P$) but 
otherwise satisfies Property~$\Delta_k$, 
we say that $P$ satisfies the \emph{weak Property~$\Delta_k$}.

\begin{mainthm}[Corollary~\ref{cor:topological-interpretation}]
    \label{mainthm:connectivity}
    Let $k\in \{1,2\}$. The set of all configurations which satisfy the
    weak Property~$\Delta_k$ forms a connected subspace of the configuration 
    space of $n$ points in $\C$.
\end{mainthm}

We are unaware of any prior appearances of the space
described in Theorem~\ref{mainthm:connectivity}. It would be interesting to 
know the homology of this space for each $k$ and, in particular, whether it 
is a classifying space for the $n$-strand braid group.

The article is structured as follows. In 
Section~\ref{sec:noncrossing-partitions}, we introduce some 
background on posets and partitions, along with basic 
properties of $\NC(P)$. Section~\ref{sec:configurations} concerns the 
transformation of configurations with Property~$\Delta_k$ and includes the proof of
Theorem~\ref{mainthm:connectivity}. We provide some technical details on ``skewers'' in
Section~\ref{sec:skewers}, then give the proofs of Theorem~\ref{mainthm:catalan-counting} in 
Section~\ref{sec:delta-1} and Theorem~\ref{mainthm:rank-symmetry} in Section~\ref{sec:delta-2}.

\section{Noncrossing Partitions}
\label{sec:noncrossing-partitions}

To start, we establish some basic definitions and properties for partitions,
posets, and configurations - see \cite{stanley-enum-vol-1} for a
standard reference.
Recall that a \emph{partition} expresses a set $S$ as the union of a collection 
of pairwise disjoint subsets of $S$ (called \emph{blocks}). The set of all
partitions for a fixed set $S$ forms a partially ordered set under refinement:
one partition lies ``below'' another in the partial order if each block in the
latter partition can be obtained as a union of blocks in the former. This
partially ordered set is a \emph{lattice} in the sense that each pair of
elements has a unique meet (greatest lower bound) and a unique join (least
upper bound). Let $\Pi(S)$ denote the lattice
of partitions for $S$; in the standard case where
$S = \{1,\ldots,n\}$, the associated partition lattice is denoted $\Pi_n$.

The partition lattice $\Pi(S)$ is \emph{bounded} in the sense that it 
contains a unique minimum element $\hat{0}$
(in which each block is a singleton) and a unique maximum element
$\hat{1}$ (in which all of $S$ belongs to the same block). The partition
lattice is also \emph{graded}: if $|S| = n$ and we let $bl(\pi)$ denote the number
of blocks in a partition $\pi \in \Pi(S)$, then the map $\rho \colon \Pi(S) \to \mathbb{N}$
given by $\rho(\pi) = n - bl(\pi)$ is a \emph{rank function} for this lattice.
Note that the minimum $\hat{0}$ and maximum 
$\hat{1}$ have ranks $0$ and $n-1$ respectively. The \emph{atoms} and \emph{coatoms} 
of this lattice are defined to be the elements of rank $1$ and $n-2$ respectively.

Our main object of study in this article is a subposet of the partition lattice
for a finite set of points in the complex plane.

\begin{defn}
    Fix $P\subset \mathbb{C}$ finite. For any $A \subseteq P$,
    the \emph{convex hull} of $A$, denoted $\conv(A)$, is the smallest convex
    subset of $\mathbb{C}$ which contains $A$. Note that $\conv(A)$ is 
    a convex polygon with up to $|A|$ vertices. A partition of $P$ is
    \emph{noncrossing} if the convex hulls of its blocks are pairwise disjoint.
    The set of all noncrossing partitions for $P$ forms a subposet of the
    partition lattice $\Pi(P)$, and we refer to this subposet as $\NC(P)$.
\end{defn}

\begin{example}\label{ex:ncp-lattice}
    Let $P = \{z_1,z_2,z_3,z_4\}$ be a set of points in $\mathbb{C}$ such that 
    $z_1$, $z_2$, and $z_3$ form the vertices of a triangle which contains $z_4$
    in its interior. Then every partition of $P$ is noncrossing except for
    $\{\{z_1,z_2,z_3\},\{z_4\}\}$, so the noncrossing partition lattice 
    $\NC(P)$ has 14 elements, arranged according to the diagram in 
    Figure~\ref{fig:ncp-ex-1}. Note that $\NC(P)$ has the same size as the
    classical noncrossing partition lattice $\NC_4$, but the two are not isomorphic
    since $\NC(P)$ has 15 maximal chains, while $\NC_4$ has 16.
\end{example}

\begin{figure}
	\centering
    \includegraphics{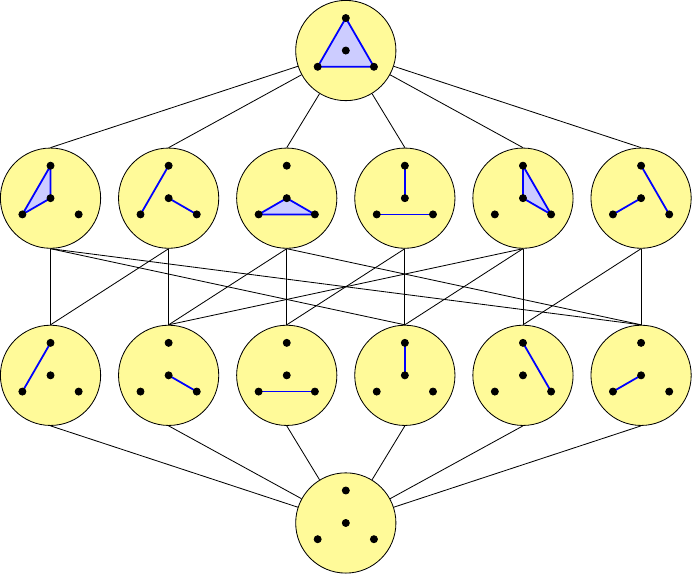}
    \caption{The lattice of noncrossing partitions for a particular arrangement of
    four points in $\mathbb{C}$}
    \label{fig:ncp-ex-1}
\end{figure}

As a poset, the noncrossing partitions of $P$ inherit some useful properties
from the larger partition lattice $\Pi(P)$.

\begin{prop}\label{prop:ncp-lattice}
    If $P\subset \C$ is finite, then $\NC(P)$ is a bounded lattice.
\end{prop}

\begin{proof}
    Since the minimum and maximum elements of $\Pi(P)$ are noncrossing, 
    we know that $\NC(P)$ is bounded. 
    To show that $\NC(P)$ is a lattice,
    we need only prove that $\NC(P)$ is a meet-semilattice (i.e. that
    each pair of elements in $\NC(P)$ has a unique meet) by a standard
    property of finite bounded posets \cite[Prop 3.3.1]{stanley-enum-vol-1}. 
    First, suppose that $\pi \leq \pi'$ in $\Pi(P)$. If $\pi'$ is noncrossing
    but $\pi$ has a pair of crossing blocks $A,B \in P$, then $\pi'$ must have
    a block which contains the union $A\cup B$. It follows that if $\pi_1,\pi_2\in \NC(P)$, 
    then the meet $\pi_1 \wedge \pi_2$ must be noncrossing; if it had a pair of 
    crossing blocks $A$ and $B$, then $\pi_1$ and $\pi_2$ would each have blocks
    containing $A\cup B$ by the argument above and thus $\pi_1 \wedge \pi_2$ would have
    a block containing $A\cup B$, a contradiction.
    Therefore, $\NC(P)$ is a meet-semilattice and thus a lattice.
\end{proof}

Although $\Pi(P)$ is graded for every configuration $P$, it is possible to construct
$P$ so that the poset $\NC(P)$ is not graded. In this article, however, we
are interested in configurations satisfying particular conditions (Properties~$\Delta_1$
and $\Delta_2$, defined in Section~\ref{sec:configurations}), and we will see in
Section~\ref{sec:skewers} that $\NC(P)$ is graded in these cases.

We close the section with a few important examples and some observations.

\begin{example}\label{ex:ncp-classical}
    If $P \subset \C$ with $|P| = n$ such that each point in $P$
    lies on the boundary of $\conv(P)$ (i.e. $P$ is in \emph{convex position}),
    then $\NC(P)$ is isomorphic to the classical \emph{noncrossing partition lattice}
    $\NC_n$, initially defined by Kreweras \cite{kreweras72} - see 
    Figure~\ref{fig:ncp-classical}. In addition to 
    the properties outlined in Proposition~\ref{prop:ncp-lattice}, Kreweras
    showed that the size of $\NC_n$ is equal to the \emph{Catalan number}
    $C_n = \frac{1}{n+1}\binom{2n}{n}$ and, in particular, the number of
    partitions in  $\NC_n$ with $k$ blocks is the \emph{Narayana number}
    $N_{n,k} = \frac{1}{n} \binom{n}{k} \binom{n}{k-1}$. For more information
    on the combinatorial significance of these connections, see \cite{stanley-catalan}.
\end{example}

\begin{figure}
	\centering
    \includegraphics{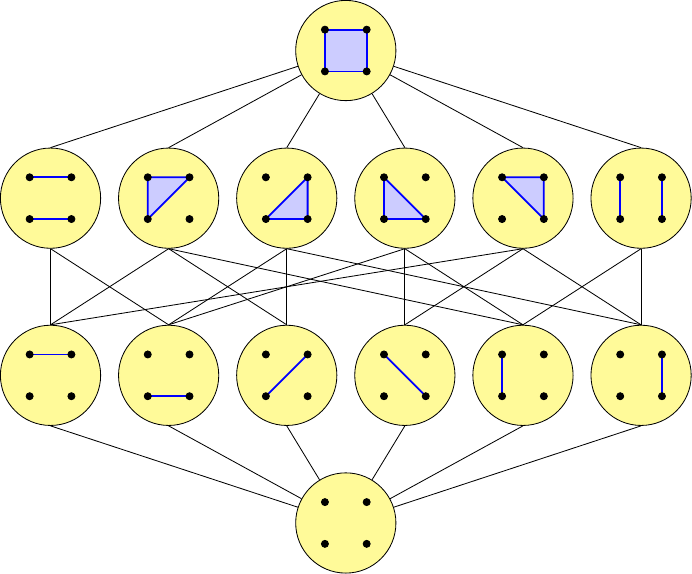}
    \caption{The classical noncrossing partition lattice $\NC_4$}
    \label{fig:ncp-classical}
\end{figure}

Noting that $N_{n,k} = N_{n,n-k}$ for all $1\leq k \leq n$, one can see
that $\NC_n$ is a \emph{rank-symmetric} lattice. In fact, the classical 
noncrossing partition lattice is \emph{self-dual} in the sense that it admits 
a bijection $\alpha\colon\NC_n \to \NC_n$ with the property that 
$\pi_1 \leq \pi_2$ if and only if $\alpha(\pi_2) \leq \alpha(\pi_1)$
\cite{simion-ullman91}. However, this stronger condition is rarely held
by $\NC(P)$ more generally.

\begin{rem}\label{rem:not-isomorphic}
    Nica and Speicher showed in 1997 that intervals in the noncrossing
    partition lattice $\NC_n$ are isomorphic to products of smaller
    noncrossing partition lattices \cite{nica-speicher}. With this in mind, 
    we note that if $P$ is a set of $n$ points of $\C$ in general position
    and if some point of $P$ lies in the convex hull of the others, then
    $\NC(P)$ has an interval which is isomorphic to the lattice described
    in Example~\ref{ex:ncp-lattice}, which cannot be expressed as a
    product of noncrossing partition lattices. Therefore, $\NC(P)$ is
    only isomorphic to $\NC_n$ if the elements of $P$ form the vertices
    of a convex $n$-gon.
\end{rem}

\begin{example}\label{ex:ncp-boolean}
    If $P\subset \C$ with $|P| = n$ such that all points in $P$ are
    collinear, then $\NC(P)$ is isomorphic to the \emph{Boolean lattice}
    $\bool_{n-1}$, which is defined as the set of all subsets of a set
    with $n-1$ elements, partially ordered under inclusion. To see this, 
    observe that each partition in $\NC(P)$ is determined precisely 
    by choosing a subset of the $n-1$ gaps between the $n$ points; 
    see Figure~\ref{fig:ncp-boolean}. 
\end{example}

\begin{figure}
	\centering
    \includegraphics{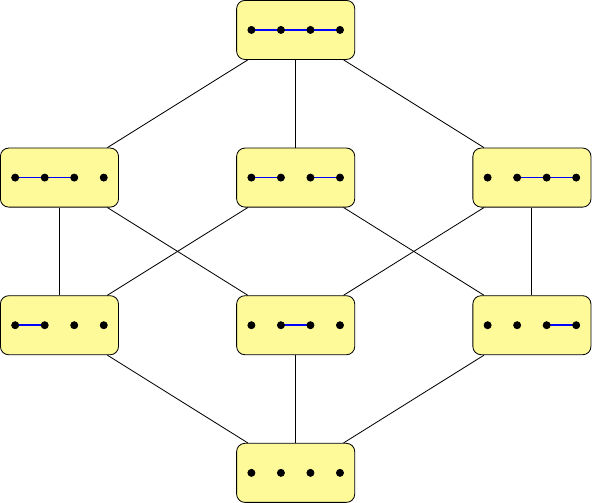}
    \caption{The Boolean lattice $\bool_{n-1}$ arises as the set of noncrossing partitions
    for a configuration of $n$ collinear points.}
    \label{fig:ncp-boolean}
\end{figure}

When $|P|=4$, there are only four possibilities for $\NC(P)$ (up to
isomorphism), and three of them are depicted in the preceding figures.
All three (indeed, all four) possess several useful lattice properties,
including rank-symmetry, self-duality, and a simple counting formula.
However, these properties do not always hold for larger sizes of $P$.

\begin{example}
    If $P$ consists of five points in general position with three points
    on the boundary of the convex hull and two points in the interior,
    then $|\NC(P)| = 43$ (whereas $|\NC_n| = 42$), although $\NC(P)$ 
    remains rank-symmetric. Furthermore, if $P$ consists of six
    points in general position, arranged so that the three extremal points 
    form an equilateral triangle and the three interior points form a
    shrunken equilateral triangle with the same center, then $\NC(P)$
    is not rank-symmetric: it has 15 atoms (rank 1) and coatoms (rank 4), but 
    55 elements at rank 2 and 57 elements at rank 3.
\end{example}

\section{Configurations}
\label{sec:configurations}

Before moving on to the main theorems, we introduce some tools for studying
the geometry of planar configurations, by which we mean finite unordered
sets of points in the Euclidean plane. We begin with some helpful terminology,
partially inspired by \cite{edelman-reiner}. Throughout this section, let $P$ 
denote a configuration of $n$ points in $\mathbb{C}$ in general position 
(i.e. no three points in $P$ are collinear), unless otherwise specified. 

\begin{defn}
    Let $A \subseteq P$ and recall that $\conv(A)$ denotes the convex hull
    of the points in $A$. Define the
    \emph{closure} $\overline{A}$ by $\conv(A) \cap P$
    and the \emph{interior} $\interior(A)$ to be $\interior(\conv(A)) \cap P$.
    We say that $A$ is \emph{convex} if $\interior(A)\cap A$ is empty. 
    Also, a point $p\in P$ is \emph{internal} if $p\in\interior(P)$ and
    \emph{extremal} otherwise. 
\end{defn}

\begin{defn}
    A configuration
    $P$ in general position satisfies \emph{Property~$\Delta_k$} if, 
    for every convex subset $A$ in $P$, the interior $\interior(A)$
    contains at most $(|A|-3) + k$ points.   
    Equivalently, $P$ has Property~$\Delta_k$ if $P$ is in general position and
    each subset $B\subseteq P$ (not necessarily convex) has at most 
    $\lfloor \frac{|B|-3+k}{2} \rfloor$ internal points. Finally, we say that $P$ instead has the 
    \emph{weak Property~$\Delta_k$} if it satisfies the same convexity criteria,
    but has at most one instance of three collinear points.
\end{defn}

It is worth noting that Property~$\Delta_1$ is equivalent to a simpler 
condition which is easier to check: for any $A\subseteq P$ with 
$|A|=3$, we have $|\interior(A)| \leq 1$. To see this, consider that
each convex subset of $m$ points in $P$ forms the vertices of a convex 
$m$-gon, and any triangulation of this polygon consists of $m-2$ triangles;
Property~$\Delta_1$ is equivalent to the requirement that each
of those $m-2$ triangles has at most one point of $P$ in its interior.
Unfortunately, this does not generalize to Property~$\Delta_k$ when
$k > 1$.

If $P$ satisfies Property~$\Delta_k$, then any small perturbation of $P$
will also satisfy Property~$\Delta_k$ since $P$ is in general position. 
However, deformations which move a point in $P$ across the line between two
other points in $P$ might not preserve this property. The main goal of this 
section is to provide some tools for moving points in $P$ while preserving
Property~$\Delta_k$. To start, we establish some terminology for the lines
connecting points in $P$.

\begin{figure}
    \centering
    \includegraphics[width=\textwidth]{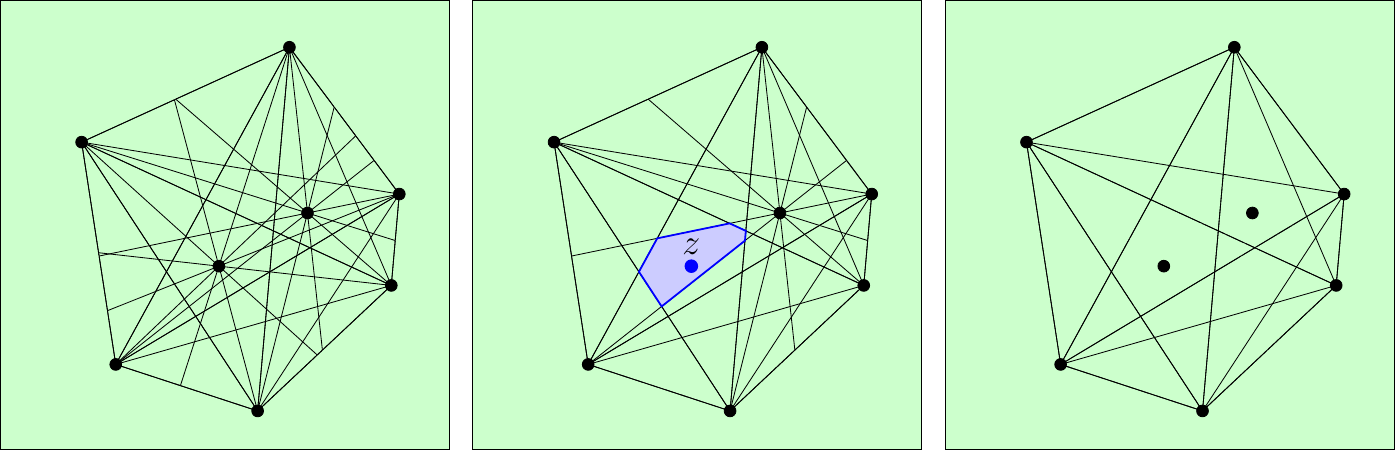}
    \caption{From left to right: an arrangement $\mathcal{A}$, 
    the subarrangement $\mathcal{A}^z$ with the region $R_z$ highlighted,
    and the subarrangement $\aext$ of lines between pairs of extremal points.
    In each image, only the core of each line has been drawn.}
    \label{fig:z-region}
\end{figure}

\begin{defn}
    Each pair of distinct points in $P$
    determines a line in $\mathbb{C}$; let $\mathcal{A}$ denote the
    arrangement of the $\binom{n}{2}$ lines obtained in this way.
    If $\ell$ is the line obtained from the points $z$ and $w$ in $P$, then 
    we say that $z$ and $w$ are the \emph{endpoints} of $\ell$
    and write $V(\ell) = \{z,w\}$. We also refer to the line segment
    between $z$ and $w$ as the \emph{core} $c(\ell) = \conv(V(\ell))$ of $\ell$.
    More generally, we write $V(\ell_1,\ldots,\ell_k)$ to mean the $2k$-element 
    set of endpoints belonging to the lines $\ell_1,\ldots,\ell_k$ and  
    we write $c(\ell_1,\ldots,\ell_k)$ to mean the convex hull $\conv(V(\ell_1,\ldots,\ell_k))$.
\end{defn}

\begin{defn}    
    For each $z$ in $\interior(P)$, let $\mathcal{A}^z$ denote the 
    subarrangement of $\mathcal{A}$ obtained by deleting the lines which pass through
    $z$. Also, define $\aext$ to be the subarrangement of $\mathcal{A}$ which
    consists of all lines with two extremal endpoints, i.e. the intersection
    of all $\mathcal{A}^z$ for $z\in \interior(P)$.
    We associate two regions to each $z\in \interior(P)$: the connected component
    of $\mathbb{C}-\mathcal{A}^z$ containing $z$ (which we denote $R_z$) and the
    connected component of $\mathbb{C} - \aext$ containing $z$
    (denoted $R_z^\text{ex}$). Note that each region is a convex polygon since
    it is a bounded subset of the plane determined by removing some number of half-planes.
    Finally, we say that a line in $\mathcal{A}$ is \emph{separating} if
    it has points from $\interior(P)$ on either side of it, and
    a \emph{boundary} line is one which contains a boundary
    edge for the convex hull $\conv(P)$. 
\end{defn}

For the sake of clarity, we will typically illustrate the
line arrangement $\mathcal{A}$ by its intersection
with the convex hull of $P$ - see Figure~\ref{fig:z-region}
for an example.

\begin{defn}\label{def:move}
    A \emph{move} is a bijection $m\colon P \to m(P)$ such that $m$
    fixes all of $P$ except some element $z$, which is instead sent
    to a point $m(z)$ in the interior of a region
    adjacent to $R_z$. If $\ell$ is a line in the arrangement $\mathcal{A}$
    which separates the regions $R_z$ and $R_{m(z)}$, then we say
    that $m$ \emph{moves $z$ across $\ell$}. If both $P$ and $m(P)$ satisfy Property 
    $\Delta_k$, we say that $m$ is a \emph{$\Delta_k$-move}.
    Finally, note that $m$ induces an isomorphism $m_*\colon \Pi(P) \to \Pi(m(P))$ by 
    replacing $z$ with $m(z)$ in each partition.
\end{defn}

It is worth noting that for any $z\in P$, we can replace $z$ with any 
other point in the region $R_z$ without changing the isomorphism type of
$\Pi(P)$, so moves on $P$ can be described solely by the regions involved.

\begin{defn}
    Let $z\in P$. If $\ell$ is 
    a line in the arrangement $\mathcal{A}$ which contains a side of the region
    $R_z$, then we say that $\ell$ is \emph{adjacent} to $z$.
    This line determines two open half-planes: $H^+_{z,\ell}$, which contains
    $z$, and $H^-_{z,\ell}$, which does not. 
\end{defn}

\begin{figure}
    \centering
    \includegraphics[width=0.5\textwidth]{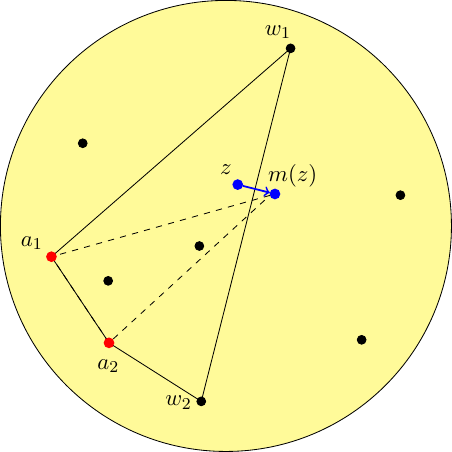}
    \caption{If $m\colon P \to m(P)$ is a move which takes $z$ across a non-separating 
    edge between extremal points $w_1$ and $w_2$, and if the triple $\{a_1,a_2,m(z)\}$ 
    (indicated with dashed lines) has more than one point in its interior, then the 
    quadrilateral $\{a_1,a_2,w_1,w_2\}$ (indicated with solid lines) has
    at least three points in its interior.}
    \label{fig:non-separating-move}
\end{figure}

The following lemmas establish two useful cases of $\Delta_k$-preserving moves.

\begin{lem}\label{lem:delta-k-move}
    If $P$ has Property~$\Delta_k$ and $m\colon P \to m(P)$ moves 
    $z\in \interior(P)$ across a non-separating line in $\aext$, then
    $m(P)$ has Property~$\Delta_k$ as well.
\end{lem}

\begin{proof}
    Let $m\colon P \to m(P)$ be a move which takes $z$ across a
    line $\ell$ in $\aext$, and let $w_1$ and $w_2$ be the endpoints of $\ell$. 
    Suppose for the sake of contradiction that $m(P)$ does not satisfy
    Property~$\Delta_k$; then there is a subset $A \subseteq m(P)$
    with $|\interior(A)| > \lfloor\frac{|A|-3+k}{2}\rfloor$. Since 
    $P$ satisfies Property~$\Delta_k$, we know that $\overline{A}$ 
    must contain $m(z)$ but not $z$.

    First, note that the supposed bound on $|\interior(A)|$ precludes $m(z)$ from
    being an internal point of $A$. If it were internal, then $A$ would necessarily contain
    $w_1$ and $w_2$, and $A$ would therefore be a subset of $H^+_{m(z),\ell}$
    since $z \not\in \overline{A}$. However, $m(z)$ is the unique internal point of 
    $m(P)$ in $H^+_{m(z),\ell}$ since we assumed $\ell$ was non-separating in the initial 
    configuration $P$, so the inequality would not hold. 

    Thus, $m(z)$ is an extremal point of $A$. Define $B = (A-\{m(z)\})\cup \{z,w_1,w_2\}$; then
    $|B| \leq |A| + 2$ and $|\interior(B)| \geq |\interior(A)| + 1$ (since $z$ is an internal
    point for $B$ but not $A$), and we can combine these to find the following chain of inequalities:
    \[
        |\interior(B)| \geq |\interior(A)| + 1 > \left\lfloor\frac{|A|-3+k}{2}\right\rfloor + 1
        \geq \left\lfloor\frac{|A|-3+k+2}{2}\right\rfloor \geq \left\lfloor\frac{|B|-3+k}{2}\right\rfloor.
    \]
    Since $B$ is a subset of $P$, this contradicts our assumption that $P$ has
    Property~$\Delta_k$ - see Figure~\ref{fig:non-separating-move} for an
    example when $k=1$. Therefore, $m(P)$ must satisfy Property~$\Delta_k$ and we are done.
\end{proof}

\begin{figure}
    \centering
    \includegraphics[width=0.5\textwidth]{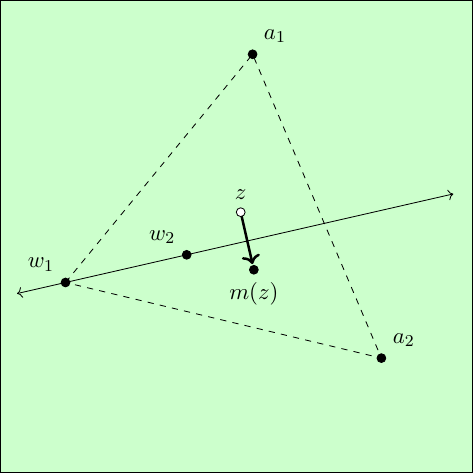}
    \caption{If $w_2$ is an internal point and $m$ is a move which takes $z$ across the line
    containing $w_1$ and $w_2$, then there are extremal points $a_1$ and $a_2$ such that the
    convex hull of $a_1$, $a_2$, and $w_1$ (depicted with dashed lines) 
    has $w_2$, $z$, and $m(z)$ in its interior.}
    \label{fig:internal-endpoint}
\end{figure}

\begin{lem}\label{lem:delta-k-internal-endpoint}
    If $P$ has Property~$\Delta_k$ and $m\colon P \to m(P)$ moves $z\in \interior(P)$
    across a line in $\mathcal{A}^z$ with at least one internal endpoint,
    then $m(P)$ has Property~$\Delta_k$ as well.
\end{lem}

\begin{proof}
    Let $\ell$ be a line in $\mathcal{A}^z$ adjacent to $z$ 
    with endpoints $w_1$ and $w_2$, suppose that $w_2$ is an internal point of $P$,
    and let $m\colon P\to m(P)$ be the move which takes $z$ across $\ell$. As
    above, we suppose for the sake of contradiction that $m(P)$ does not satisfy
    Property~$\Delta_k$ and can thus find a subset $A\subseteq m(P)$ with
    $|\interior(A)| > \lfloor\frac{|A|-3+k}{2}\rfloor$ such that $\overline{A}$ contains
    $m(z)$ but not $z$.

    Consider the three lines in $\mathcal{A}$ for which one endpoint is $w_1$ and the
    other belongs to the set $\{z,w_2,m(z)\}$. Since $z$, $w_2$, and $m(z)$ are internal
    points of $P$ and both $z$ and $m(z)$ are adjacent to $\ell$, all three of these lines
    must pass through the same side of the polygon $\conv(P)$. Let $a_1,a_2\in P$ be the
    extremal points which determine this side, where $a_1$ is on the same side of $\ell$
    as $z$ - see Figure~\ref{fig:internal-endpoint} for an illustration.

    Now, define $B = (A-\{m(z)\}) \cup \{z,w_1,w_2,a_1,a_2\}$. If $m(z)\in \interior(A)$, 
    then $A$ must contain $w_1$ and $w_2$, and the fact that 
    $z\notin A$ tells us that $A$ does not contain any points in the half-plane $H^+_{z,\ell}$.
    Thus, in this case we have that $|B| \leq |A| + 2$ and $|\interior(B)| \geq |\interior(A)| + 1$
    (since $w_2$ is internal for $B$ but not $A$),
    which provides the same chain of inequalities as described in the proof of Lemma~\ref{lem:delta-k-move}.
    Therefore, $P$ does not satisfy Property~$\Delta_k$, which is a contradiction.

    If $m(z)$ is instead an extremal point of $A$, then we see that $|B| \leq |A| + 4$ and
    $|\interior(B)| \geq |\interior(A)| + 2$, so we have a similar sequence of inequalities:
    \[
        |\interior(B)| \geq |\interior(A)| + 2 > \left\lfloor\frac{|A|-3+k}{2}\right\rfloor + 2
        \geq \left\lfloor\frac{|A|-3+k+4}{2}\right\rfloor \geq \left\lfloor\frac{|B|-3+k}{2}\right\rfloor.
    \]
    This also contradicts our assumption that $P$ satisfies Property~$\Delta_k$, so we
    conclude that $m(P)$ must satisfy Property~$\Delta_k$.
\end{proof}

\begin{figure}
    \centering
    \includegraphics[width=\textwidth]{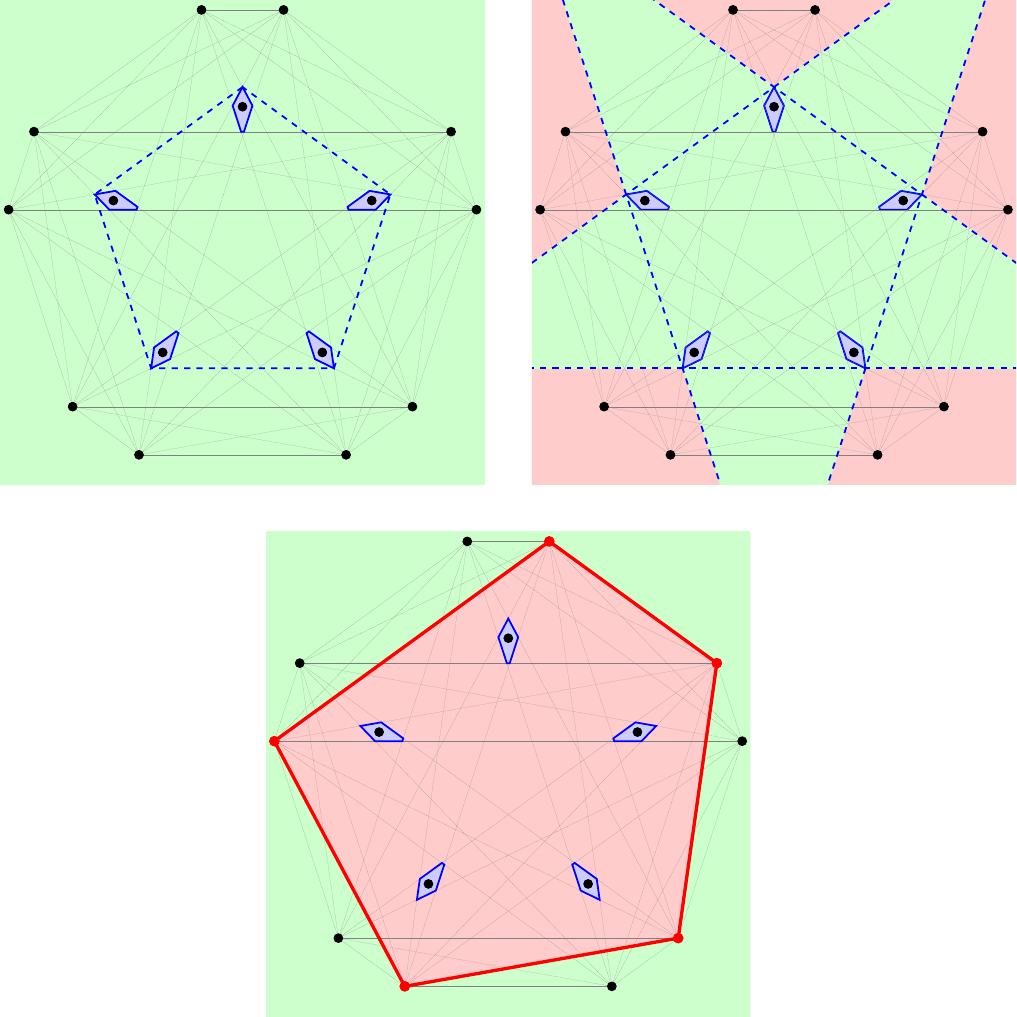}
    \caption{This configuration of 15 points has five internal points for
    which the corresponding regions have a convex hull (outlined with 
    dashed blue lines in the upper left image) where no side of the convex hull contains a 
    side of a region. By extending each side of the convex hull into a line, 
    each of the five points determines
    a cone (shaded red in the upper right image) which contains at least one extremal point. 
    Selecting one extremal point from each cone (highlighted red in the bottom image)
    yields a set of extremal points which contains the starting internal points
    and is at most as numerous, thus violating Property~$\Delta_2$.}
    \label{fig:adjacent-non-separating}
\end{figure}

Note that the following lemma supposes only that $P$ has
Property~$\Delta_2$, so in particular it holds when $P$ satisfies
Property~$\Delta_1$ as well.

\begin{lem}\label{lem:adjacent-non-separating}
    If $P$ has Property $\Delta_2$, then there is a point
    $z\in \interior(P)$ such that at least one side of the region $R_z^\text{ex}$ 
    belongs to a non-separating line in $\aext$.
\end{lem}

\begin{proof}
    Let $P$ be a configuration satisfying Property $\Delta_2$ and
    let $D$ be the convex hull of the regions $R_z^\text{ex}$, where $z$ is an
    interior point of $P$. Let $\ell_1,\ldots,\ell_k$ denote the lines
    (not necessarily in $\mathcal{A}$)
    which contain the $k$ sides of $D$, arranged so that they appear in
    counter-clockwise order. If at least one $\ell_i$ contains a side of
    some region $R_z^\text{ex}$, then $\ell_i$ belongs to the arrangement 
    $\aext$, and it follows that $\ell_i$ must be non-separating since all 
    internal points lie on one side of it.

    Suppose for the sake of contradiction that this is not the case.
    Then we can fix points $z_1,\ldots,z_k \in \interior(P)$ such that
    for each $i$, the region $R_{z_i}$ intersects the boundary of $D$
    at the point where the lines $\ell_i$ and $\ell_{i+1}$ (evaluated
    mod $k$) intersect. Since neither $\ell_i$ nor $\ell_{i+1}$ are
    in $\aext$, it follows that the half-planes $H^-_{z_i,\ell_i}$ and
    $H^-_{z_i,\ell_{i+1}}$ intersect in an unbounded region which contains
    an extremal point of $P$ - see Figure~\ref{fig:adjacent-non-separating} for an illustration.

    If $E$ is the set of $k$ extremal points obtained in the manner above,
    then one can show that the convex hull $\conv(E)$ contains $D$. However,
    this means that $z_1,\ldots,z_k$ lie in the interior of $E$, which
    contradicts our assumption that $P$ satisfies Property~$\Delta_2$.
    Therefore, at least one side of $D$ must belong to a non-separating
    line in $\aext$ and we are done.
\end{proof}

Putting all of these tools together, we obtain a useful connectivity property.

\begin{thm}\label{thm:delta-k-connectivity}
    Suppose $P$ satisfies Property~$\Delta_k$, where $k\in \{1,2\}$. Then there is
    a sequence of $\Delta_k$-moves which transforms $P$ into a convex configuration.
\end{thm}

\begin{proof}
    We proceed by induction on the number of internal points for $P$.
    If $P$ has no internal points, then $P$ is already convex and we are done.
    Now, suppose the theorem holds for configurations with fewer internal points 
    than $P$ and define $\text{nsnb}(P)$ to be the set of non-separating non-boundary
    lines in $\aext$; we will prove that the theorem holds for $P$ as well by 
    a second induction on $|\text{nsnb}(P)|$. If $|\text{nsnb}(P)| = 0$, then each
    internal point of $P$ is adjacent to a boundary line, and by 
    Lemma~\ref{lem:delta-k-move}, we can bring one of these internal points
    across a boundary line by a $\Delta_k$-move, which reduces the number of internal
    points by one and allows us to apply the first inductive hypothesis.

    Next, suppose the claim holds for all configurations $Q$ with $|\interior(Q)| = |\interior(P)|$
    and $|\text{nsnb}(Q)| < |\text{nsnb}(P)|$. By Lemma~\ref{lem:adjacent-non-separating},
    we know that there is an internal point
    $z$ in $P$ such that one side of the region $R_z^\text{ex}$ is contained in
    a non-separating line in $\aext$. This means that while $z$ may not
    be adjacent to a non-separating line, there is a finite sequence of moves
    across lines with at least one internal endpoint which takes $z$ to a region
    which is adjacent to a non-separating line in $\aext$.
    Each move across lines with an internal endpoint preserves Property~$\Delta_k$ by
    Lemma~\ref{lem:delta-k-internal-endpoint}, after which we can perform a $\Delta_k$-move across the 
    non-separating line by Lemma~\ref{lem:delta-k-move}. This new configuration has 
    fewer non-separating non-boundary lines, so by the second inductive hypothesis, 
    it can be further transformed via $\Delta_k$-moves into
    a convex configuration and the proof is complete.
\end{proof}

From a topological perspective, Theorem~\ref{thm:delta-k-connectivity} can be 
interpreted as a statement about the configuration space of $n$ points in 
the plane. To do so, recall that the 
\emph{weak Property~$\Delta_k$} is a slight weakening of Property~$\Delta_k$
which allows for one instance of three collinear points. 

\begin{cor}[Theorem~\ref{mainthm:connectivity}]
    \label{cor:topological-interpretation}
    Let $k\in \{1,2\}$. The set of all configurations which satisfy the
    weak Property~$\Delta_k$ forms a connected subspace of the configuration 
    space of $n$ points in $\C$.
\end{cor}

Applying a move to a configuration $P$ will certainly affect the 
noncrossing partition lattice $\NC(P)$, and possibly even its 
isomorphism type. To close this section, we introduce
a natural map between the larger partition lattices.

\begin{figure}
    \centering
    \includegraphics[width=0.9\textwidth]{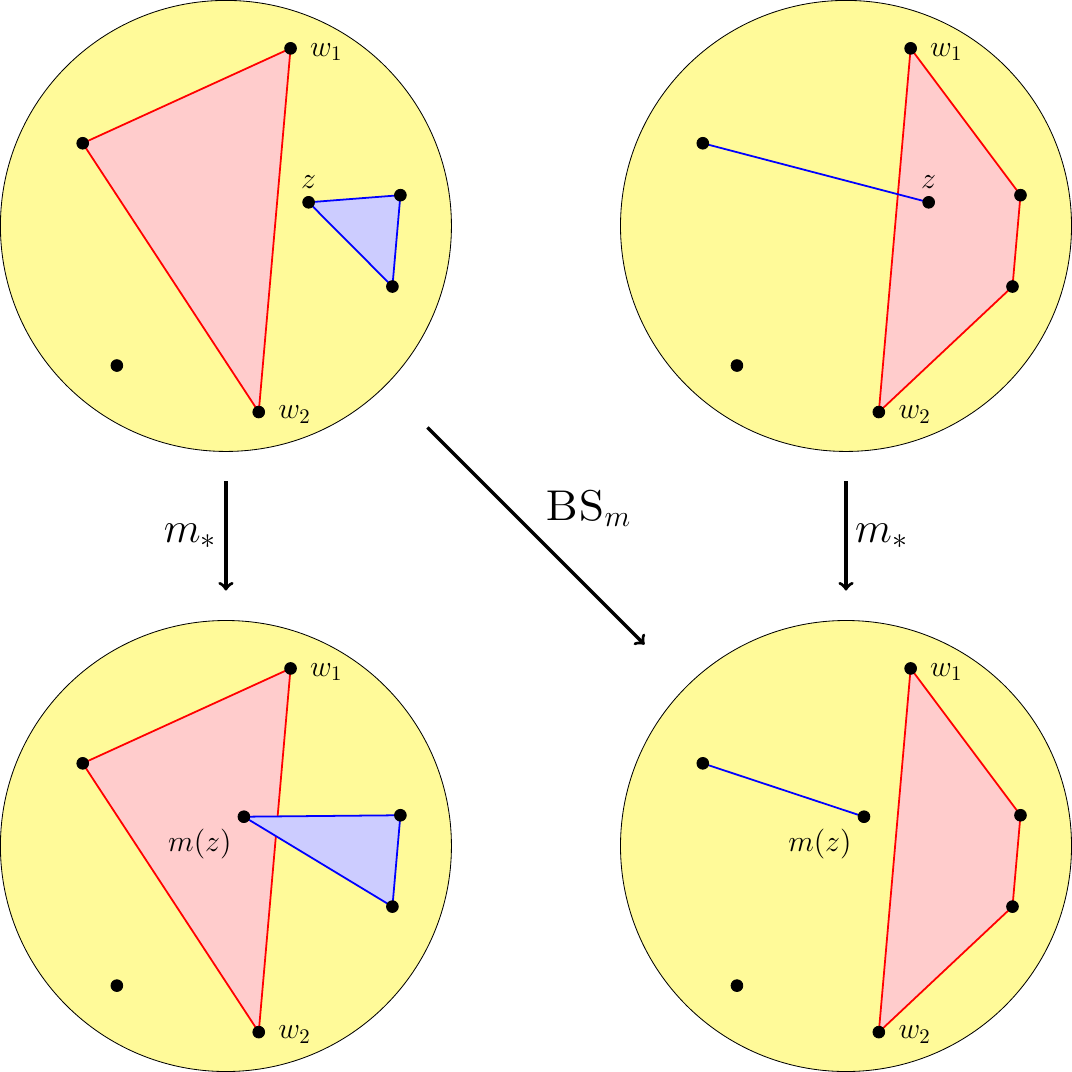}
    \caption{The block-switching map $\BS_m$, compared to the 
    induced map $m_*$ for a fixed move $m$. In this example, $\BS_m$
    takes an element of $\NC(P)$ to an element of $\NC(m(P))$.}
    \label{fig:block-switching}
\end{figure}

\begin{defn}\label{def:block-switching}
    Let $m\colon P \to m(P)$ be a move which brings the point $z\in P$
    across the line $\ell$ in $\mathcal{A}^z$ and let $w_1$ and $w_2$ 
    be the two points of $P$ on $\ell$. 
    The \emph{block-switching map} 
    $\BS_m\colon \Pi(P) \to \Pi(m(P))$ is defined for each $\pi \in \Pi(P)$ 
    as follows: if $w_1$ and $w_2$ share a block in $\pi$ and $z$ belongs
    to a different block, then define 
    $\BS_m(\pi)$ to be the result of removing $\{w_1,w_2\}$
    and $\{m(z)\}$ from their respective blocks in $m_*(\pi)$ and swapping them;
    otherwise, define $\BS_m(\pi) = m_*(\pi)$. See Figure~\ref{fig:block-switching}
    for an illustration.
    Note that $\BS_m$ is a rank-preserving bijection, but not an isomorphism.
\end{defn}

\section{Skewers}
\label{sec:skewers}

In this section, we introduce the notion of ``skewering'' in the line arrangement
$\mathcal{A}$ and prove some technical results which come from Property~$\Delta_2$. 
These tools are used in this section to prove that $\NC(P)$ is graded when $P$ 
satisfies Property~$\Delta_2$ (Proposition~\ref{prop:graded}) and in Section~\ref{sec:delta-2} to prove
Theorem~\ref{mainthm:rank-symmetry}.
Throughout the rest of this section,
let $P$ be a configuration of $n$ points in $\C$ and let $\mathcal{A}$ be the
corresponding line arrangement.

\begin{defn}
    For each pair of distinct lines 
    $\ell_1,\ell_2 \in \mathcal{A}$, the intersection $\ell_1\cap\ell_2$ can 
    be classified into one of four different types (without loss of generality):
    \begin{itemize}
        \item if $\ell_1\cap\ell_2 = \emptyset$, then $\ell_1$ and $\ell_2$
        are \emph{parallel};
        \item if $\ell_1\cap\ell_2$ lies in $c(\ell_1)\cap c(\ell_2)$,
        then $\ell_1$ and $\ell_2$ \emph{intersect internally};
        \item if $\ell_1\cap\ell_2$ lies in neither $c(\ell_1)$ nor
        $c(\ell_2)$, then $\ell_1$ and $\ell_2$ \emph{intersect externally};
        \item if $\ell_1\cap\ell_2$ lies in $c(\ell_2)$ but not $c(\ell_1)$, 
        then $\ell_1$ \emph{skewers} $\ell_2$.
    \end{itemize}
    As a useful shorthand, we write $\ell_1\dashv \ell_2$ to mean that $\ell_1$ skewers $\ell_2$.
    When this is the case, note that the convex hull $c(\ell_1,\ell_2)$ contains
    one of the endpoints of $\ell_1$ as an internal point - we refer to this as the
    \emph{link vertex} for the skewer. Finally, a \emph{skewering sequence} is a collection of
    lines $\ell_1,\ldots,\ell_k$ in $\mathcal{A}$ with $\ell_1 \dashv \ell_2 \dashv \cdots \dashv \ell_k$.
\end{defn}

\begin{figure}
    \centering
    \includegraphics[width=0.5\textwidth]{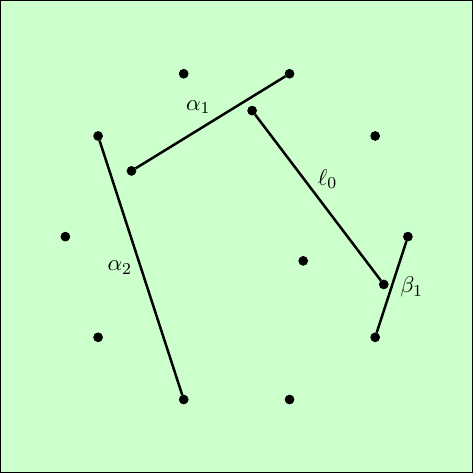}
    \caption{A skewering tree for a configuration of $14$ points which
    satisfies Property~$\Delta_2$, but not Property~$\Delta_1$. The tree
    consists of two skewering sequences: $\ell_0 \dashv \alpha_1 \dashv \alpha_2$
    and $\ell_0 \dashv \beta_1$. For the sake of visual clarity, only the
    cores of the four lines have been drawn.}
    \label{fig:skewering-tree}
\end{figure}

\begin{defn}
    Suppose that $P$ satisfies Property~$\Delta_2$. A \emph{skewering tree}
    is a subset $T \subset \mathcal{A}$, together with two additional pieces
    of data - a special element $\ell_0 \in T$ called the \emph{initial line} 
    and a chosen closed half-plane bounded by $\ell_0$, which we call the 
    \emph{positive side} of $\ell_0$ - with the following properties:
    \begin{itemize}
        \item no two elements of $T$ intersect internally; 
        \item for all $\ell \in T$ with $\ell \neq \ell_0$, there
        is a unique line $\ell' \in T$ which skewers $\ell$;
        \item no element of $P$ is the link vertex for more than one
        skewer.
    \end{itemize}
    We can also build a skewering tree inductively as follows: begin 
    with an initial line $\ell_0\in \mathcal{A}$, select one of the two half-planes
    bounded by $\ell_0$ to be the positive side, and define $T = \{\ell_0\}$.
    Next, either stop here or add a line from $\mathcal{A}$ to $T$ which is 
    skewered by another element of $T$ such that the requirements above remain 
    satisfied. Repeat this process and stop at any point; the resulting set 
    $T$ is a skewering tree. See Figure~\ref{fig:skewering-tree} for an
    example. In either construction, we refer to the non-initial 
    lines in $T$ which do not skewer any other lines as \emph{leaves}. Finally,
    we say that a skewering tree is \emph{maximal} if it is not properly contained in
    any other skewering tree in $\mathcal{A}$.
\end{defn}

\begin{figure}
    \centering
    \includegraphics[width=0.5\textwidth]{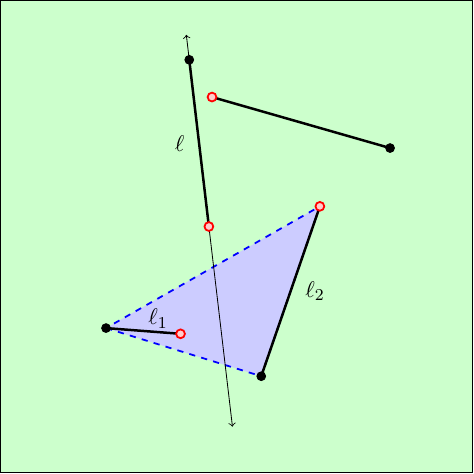}
    \caption{If $\ell$ intersects the convex hull of $\ell_1$ and $\ell_2$
    (denoted in blue), then the set of all endpoints has half of its elements in the
    interior (shown here as unfilled red dots).}
    \label{fig:no-skewer-cycles}
\end{figure}

The following lemma places fairly strong restrictions on the planar structure
of skewering trees in configurations which satisfy Property~$\Delta_2$.

\begin{lem}\label{lem:no-skewer-overlap}
    Suppose $P$ satisfies Property~$\Delta_2$, let $T \subset \mathcal{A}$ 
    be a skewering tree, and let $\ell$, $\ell_1$, and $\ell_2$ be distinct
    lines in $T$ such that $\ell$ is a leaf and $\ell_1 \dashv \ell_2$. 
    Then $\ell$ does not intersect the convex hull $c(\ell_1,\ell_2)$.
\end{lem}

\begin{proof}
    First, suppose that $\ell$ is the last element in a skewering sequence
    which includes $\ell_1$ and $\ell_2$, i.e. that there is a skewering 
    sequence $\ell_0 \dashv \alpha_1 \dashv \cdots \alpha_k$ such that
    $\ell = \alpha_k$, and $\alpha_i = \ell_1$ and $\alpha_{i+1} = \ell_2$
    for some $i$. The set $V(\{\ell_0,\alpha_1,\ldots,\alpha_k\})$ then
    consists of $2(k+1)$ points, of which at least $k$ are internal: one
    for each non-leaf element in the skewering sequence. 
    If $\ell$ were to intersect the convex hull $c(\ell_1,\ell_2)$, then
    one of the endpoints of $\ell$ would lie in the interior of the convex
    hull $c(\ell_1,\ell_2,\ell)$, but that would mean that the $2(k+1)$-element
    set described above has $k+1$ internal points, which contradicts our
    assumption that $P$ satisfies Property~$\Delta_2$.

    On the other hand, suppose that $\ell$ is in a different skewering
    sequence than $\ell_1$ and $\ell_2$. That is, suppose that $T$ contains
    the sequences $\ell_0 \dashv \alpha_1 \dashv \cdots \dashv \alpha_k$
    and $\ell_0 \dashv \beta_1 \dashv \cdots \dashv \beta_m$, where
    $\alpha_k = \ell$, and $\beta_i = \ell_1$ and $\beta_{i+1} = \ell_2$ 
    for some $i$. Then 
    $V(\{\ell_0,\alpha_1,\ldots,\alpha_k,\beta_1,\ldots,\beta_m\})$
    is a set of $2(k+m+1)$ points in $P$, of which at least $k+m$ are internal.
    Once again, if $\ell$ intersected the convex hull $c(\ell_1,\ell_2)$,
    this would imply that our set of $2(k+m+1)$ points has $k+m+1$ internal
    points, which again would violate Property~$\Delta_2$. In both cases, 
    we see that $\ell$ does not intersect $c(\ell_1,\ell_2)$.
\end{proof}

Note that one may ``prune'' a skewering tree by iteratively removing leaves, 
and the result remains a skewering tree at each step. Thus, Lemma~\ref{lem:no-skewer-overlap}
further shows that the non-leaf elements of a skewering tree are similarly constrained.

As mentioned in Section~\ref{sec:noncrossing-partitions}, the full partition lattice 
$\Pi(P)$ is graded by the rank function $\rho(\pi) = n-bl(\pi)$ for any choice of 
configuration $P$, but it is possible to construct $P$ such that $\NC(P)$ is
not graded. For example, Figure~\ref{fig:ungraded} depicts
a noncrossing partition with three blocks which is covered by the
maximum element $\hat{1}$ in $\NC(P)$, but not in $\Pi(P)$. In the following proposition, 
we use Lemma~\ref{lem:no-skewer-overlap} to show that configurations with Property~$\Delta_2$
avoid this sort of behavior.

\begin{figure}
    \centering
    \includegraphics[width=0.5\textwidth]{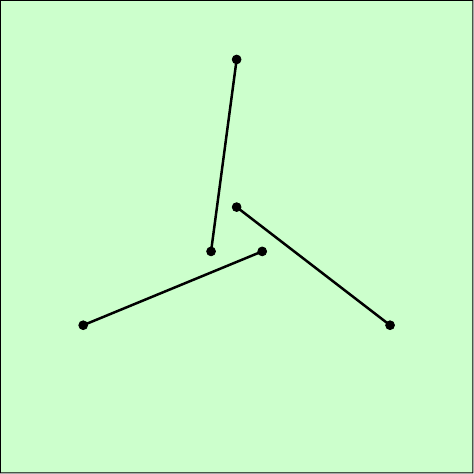}
    \caption{No two blocks of this noncrossing partition can be combined to make
    a coarser noncrossing partition without also including the third, so this
    is a partition with three blocks which is covered by the maximum element
    $\hat{1}$, which has one block.}
    \label{fig:ungraded}
\end{figure}

\begin{prop}\label{prop:graded}
	If $P$ satisfies Property~$\Delta_2$, then $\NC(P)$ is graded.
\end{prop}

\begin{proof}
	Suppose $P$ has Property~$\Delta_2$ and let $\pi < \pi'$ in $\NC(P)$
	be a covering relation, i.e. there is no $\pi'' \in \NC(P)$ with 
	$\pi < \pi'' < \pi'$. Then $\pi'$ is obtained from $\pi$ in the
	following manner: there is a collection $\mathcal{B}$ of $k \geq 2$ blocks in $\pi$
	which are removed and replaced by the union $U = \cup \mathcal{B}$ to
	create $\pi'$. If we can show that $k = 2$, then we have that 
	$\rho(\pi) + 1 = \rho(\pi')$, and therefore $\rho$ is a rank function for $\NC(P)$. 
	
	To this end, let $\mathcal{A}_{\pi}$ denote the lines in $\mathcal{A}$ which contain a side of 
	$\conv(A)$ for some block $A \in \pi$ and let $A_1 \in \mathcal{B}$. 
	Since $k\geq 2$, there is a side of $\conv(A_1)$ which does not lie in the 
	boundary of $\conv(U)$; let $\ell_1 \in \mathcal{A}_\pi$ be the line which 
	contains this side. If $\ell_1$ does not skewer any other line in $\mathcal{A}_\pi$,
	then each block in $\pi$ lies on one side of $\ell_1$ or the other,
	and therefore one may define a new partition $\pi''$ by replacing $\mathcal{B}$ with
	two blocks, each of which is the union of all blocks in $\mathcal{B}$ with
	interior on a particular side of $\ell_1$. In this case, we would have
	$\pi \leq \pi'' < \pi'$, which implies $\pi = \pi''$ by our initial assumptions,
	thus $k=2$ and we are done.
	
	Suppose $\ell_1$ does skewer another line $\ell_2 \in \mathcal{A}_\pi$. 
	Without loss of generality, we may choose $\ell_2$ to be one of the 
	``closest'' to $\ell_1$ in the sense 
	that the segment of $\ell_1$ between the intersection $\ell_1 \cap \ell_2$ and the core 
	$c(\ell_1)$ does not intersect any elements of $\mathcal{A}_\pi$. Then $\ell_2$ does
	not contain a side of $\conv(U)$, so we may repeat the analysis in the previous
	paragraph: if $\ell_2$ does not skewer any other line in $\mathcal{A}_\pi$, then 
	we are done. Otherwise, we can continue this process to obtain a skewering sequence
	$\ell_1 \dashv \ell_2 \dashv \cdots \dashv \ell_m$ such that each $\ell_i$ is an element 
	of $\mathcal{A}_\pi$ which does not contain a side of $\conv(U)$, and by 
	Lemma~\ref{lem:no-skewer-overlap}, we can guarantee that this sequence terminates in
	a line $\ell_m$ which does not skewer any line in $\mathcal{A}_\pi$. Applying the argument
	above to $\ell_m$, the proof is complete.
\end{proof}

Next, we investigate how skewering trees decompose configurations into
regions.

\begin{defn}
    The union of all lines in a skewering tree $T$ forms a subset of the plane with the
    structure of an unbounded graph, consisting of vertices, line segments,
    rays, and lines (one may equivalently view this as a graph embedding on the
    2-dimensional sphere, viewed as the stereographic projection of the plane);
    let $\Gamma_T$ denote the unbounded graph obtained from this by
    removing any ray which does not contain the core of its corresponding
    line in $T$. We refer to $\Gamma_T$ as the \emph{planar realization} of $T$,
    observing that $\Gamma_T$ is an acyclic graph, i.e. a tree. 
    Note that while a pair of rays in $\Gamma_T$ might overlap, 
    Lemma~\ref{lem:no-skewer-overlap} implies that no ray in $\Gamma_T$ has a
    transverse intersection with a line segment.
\end{defn}

\begin{figure}
    \centering
    \includegraphics[width=\textwidth]{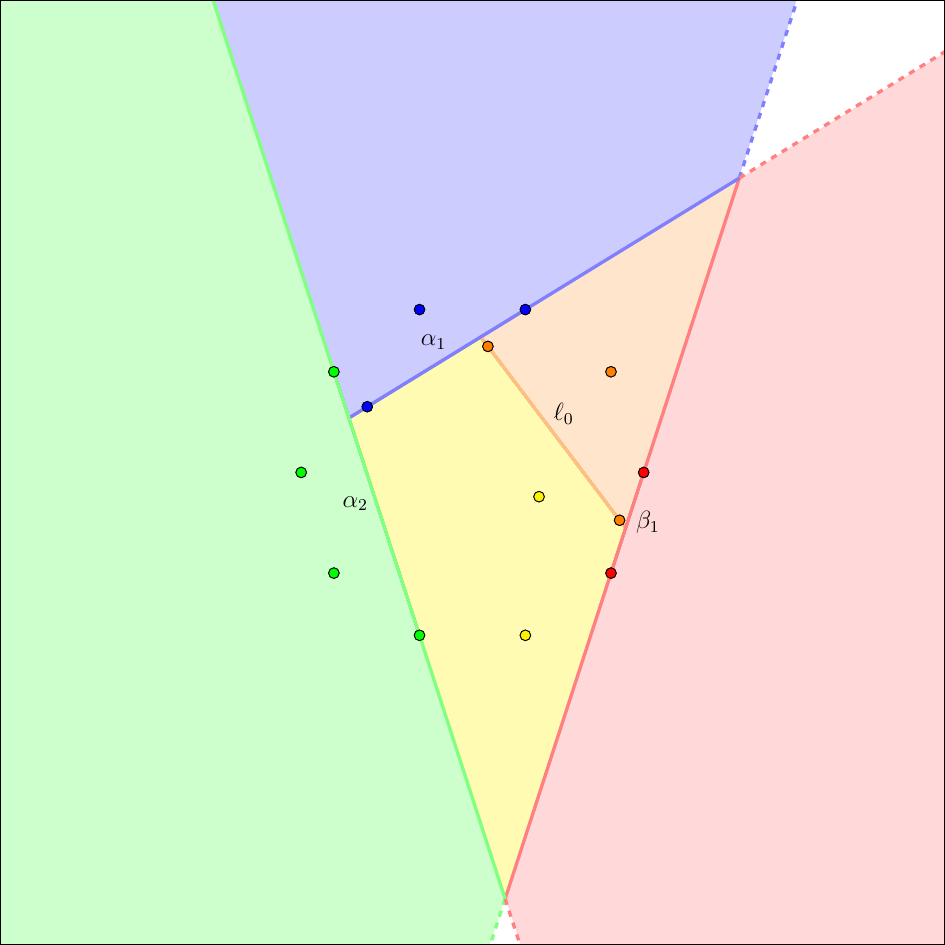}
    \caption{Cells associated to the skewering tree in 
    Figure~\ref{fig:skewering-tree}. In this case, the positive side
    of the initial line $\ell_0$ is chosen to be the upper-right side; the
    corresponding cell is shaded orange.}
    \label{fig:cells}
\end{figure}

\begin{defn}
    For each non-initial line $\ell$ in $T$, the core $c(\ell)$
    is contained in the boundary for exactly one region of the complement
    $\C - \Gamma_T$ (since the other side of $\ell$ contains the core of the line
    which skewers it). The \emph{cell} associated to $\ell$, denoted $C_\ell$, is 
    the union of the interior of this region together with the line segment or
    ray which contains $c(\ell)$. Note that $C_\ell$ is a convex (possibly unbounded)
    subset of the plane which contains exactly one side of its boundary.
    The initial line $\ell_0$ bounds two regions and thus corresponds to two
    cells: $C^+_{\ell_0}$, which contains the core $c(\ell_0)$ and belongs to
    the positive side of $\ell_0$, and $C^-_{\ell_0}$, which does not. 
    See Figure~\ref{fig:cells} for an illustration.
    For any cell $C^{\pm}_\ell$, we write $V(C^{\pm}_\ell)$ to mean the intersection 
    of $P$ with $C^{\pm}_\ell$.
\end{defn}

\begin{lem}\label{lem:disjoint-cells}
    The cells of a skewering tree are pairwise disjoint.
\end{lem}

\begin{proof}
    The interior of a cell for a skewering tree is a connected component
    of the complement $\mathbb{C}-\Gamma_T$, so it follows that the interiors of
    two cells overlap if and only if the interiors are identical. What remains to
    be shown is that no connected component of $\mathbb{C}-\Gamma_T$ belongs
    to the cores of two different lines in $T$.
    
    First, we consider lines $\alpha_i$ and $\alpha_j$ which come from the same 
    skewering sequence $\alpha_0 \dashv \alpha_1 \dashv \cdots \dashv \alpha_k$,
    where $0\leq i < j \leq k$. Note that the union
    \[
        c(\alpha_i,\alpha_{i+1}) \cup c(\alpha_{i+1},\alpha_{i+2}) \cup 
        \cdots\cup c(\alpha_{j-2},\alpha_{j-1})
    \]
    is a connected subset of the plane and by Lemma~\ref{lem:no-skewer-overlap}, it 
    cannot intersect $\alpha_j$. Thus, this subset lies on one side of $\alpha_j$
    (the side which contains the cell $C_{\alpha_i}$), while the cell $C_{\alpha_j}$ lies
    on the other. Thus, the cells associated to $\alpha_i$ and $\alpha_j$ are disjoint.
    
    Now, we consider lines $\alpha_i$ and $\beta_j$ which come from distinct 
    skewering sequences $\alpha_0 \dashv \alpha_1 \dashv \cdots \dashv \alpha_k$
    and $\beta_0 \dashv \beta_1 \dashv \cdots \dashv \beta_m$, where $\alpha_0 = \ell_0 = \beta_0$
    and $i,j \geq 1$. Similar to the previous case, we observe that
    \[
        \left(\bigcup_{t=1}^{i-1} c(\alpha_{t-1},\alpha_t)\right)
        \cup
        \left(\bigcup_{t=1}^{j-1} c(\beta_{t-1},\beta_t)\right)
    \]
    is a connected subset of the plane and by Lemma~\ref{lem:no-skewer-overlap},
    it must be wholly contained in one of the four components of the complement
    $\C - (\alpha_i\cup\beta_j)$. In particular, it must belong to the unique component
    which contains the core of $\alpha_i$ and the core of $\beta_j$ in its boundary. Thus,
    this region of $\C - (\alpha_i\cup\beta_j)$ does not contain the cells for either 
    $\alpha_i$ or $\beta_j$, which means that the two cells are disjoint - see
    Figure~\ref{fig:disjoint-cells} for an illustration.
\end{proof}

\begin{figure}
    \centering
    \includegraphics[width=0.5\textwidth]{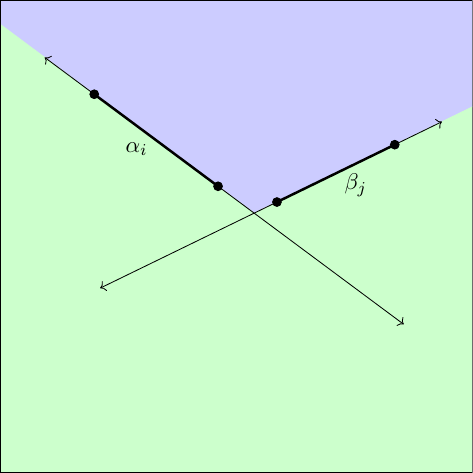}
    \caption{In the complement of the lines $\alpha_i$ and $\beta_j$, 
    the region which has the cores of $\alpha_i$ and $\beta_j$ in its boundary 
    (shaded blue here) must contain the cores of lines which skewer $\alpha_i$ and $\beta_j$,
    which means that the cells $C_{\alpha_i}$ and $C_{\beta_j}$ lie outside the shaded region.}
    \label{fig:disjoint-cells}
\end{figure}

\begin{lem}\label{lem:cells-cover}
    Let $P$ be a configuration of $n$ points which satisfies Property~$\Delta_2$.
    Then the cells of a skewering tree cover the elements of $P$.
\end{lem}

\begin{proof}
    Let $z\in P$, let $T$ be a skewering tree and let $\Omega$ be the connected 
    component of $\C - \Gamma_T$ which contains $z$. Suppose for the sake of 
    contradiction that $\Omega$ does not belong to a cell of $T$. This immediately
    rules out the possibility that $\Omega$ touches only a single line in $T$,
    since that line would necessarily be a leaf and thus $\Omega$ would belong
    to the cell associated to that line. The remaining case to consider is that
    that there are distinct lines $\ell_1, \ell_2 \in T$ such that 
    \begin{enumerate}
        \item  $\ell_1$ and $\ell_2$ bound adjacent sides of $\Omega$, and
        \item $\Omega$ lies in the unique component of $\C - (\ell_1\cup\ell_2)$ which has
        neither $c(\ell_1)$ nor $c(\ell_2)$ in its boundary.
    \end{enumerate}
    Notice that this implies that the endpoints of $\ell_1$, the endpoints of
    $\ell_2$, and $z$ form a 5-element subset of $P$ where $\ell_i$ and $\ell_j$
    each have one (non-link) endpoint in the interior of the convex hull.
    Similar to the proof of Lemma~\ref{lem:disjoint-cells}, we consider 
    two cases according to whether $\ell_1$ and $\ell_2$ belong to the 
    same skewering sequence or not.

    First, suppose $T$ contains a skewering sequence 
    $\alpha_0 \dashv \alpha_1 \dashv \cdots \dashv \alpha_k$ such that
    $\ell_1 = \alpha_i$ and $\ell_2 = \alpha_j$, where $0\leq i < j \leq k$
    and let $A$ be the set containing both $z$ and the endpoints of 
    $\alpha_i, \alpha_{i+1},\ldots, \alpha_j$. Then $|A| = 2(j-i+1) + 1$ 
    and we know that $j-i$ points in $A$ are link vertices for skewers 
    involving other lines in $A$. Together with the two points mentioned 
    above, we have that $A$ is a set of $2(j-i+1) + 1$ points where $j-i+2$ 
    of them are internal, which contradicts our assumption
    that $P$ satisfies Property~$\Delta_2$.

    Next, suppose that $T$ contains skewering sequences 
    $\alpha_0 \dashv \alpha_1 \dashv \cdots \dashv \alpha_k$
    and $\beta_0 \dashv \beta_1 \dashv \cdots \dashv \beta_m$ such that 
    $\ell_1 = \alpha_i$ and $\ell_2 = \beta_j$ for some $i,j \geq 1$, and let
    $B$ be the set containing $z$ together with the endpoints of 
    $\ell_0,\alpha_1,\ldots,\alpha_i,\beta_1,\ldots,\beta_j$. 
    Then $|B| = 2(i+j+1) + 1$, and at least $i+j$ elements of $B$ are 
    internal due to being link vertices of skewers in this set. 
    By the same reasoning as in the previous case, we know that $i+j+2$ of 
    the $2i+2j+3$ points in $B$ are internal, which again contradicts our 
    assumption that $P$ satisfies Property~$\Delta_2$.

    Since both cases lead to a contradiction of Property~$\Delta_2$, we 
    conclude that $z$ must in fact belong to a cell of the skewering tree 
    $T$, and this completes the proof.    
\end{proof}

We now conclude this section by defining a subposet of $\NC(P)$ associated to each
skewering tree.

\begin{figure}
    \centering
    \includegraphics[width=0.8\textwidth]{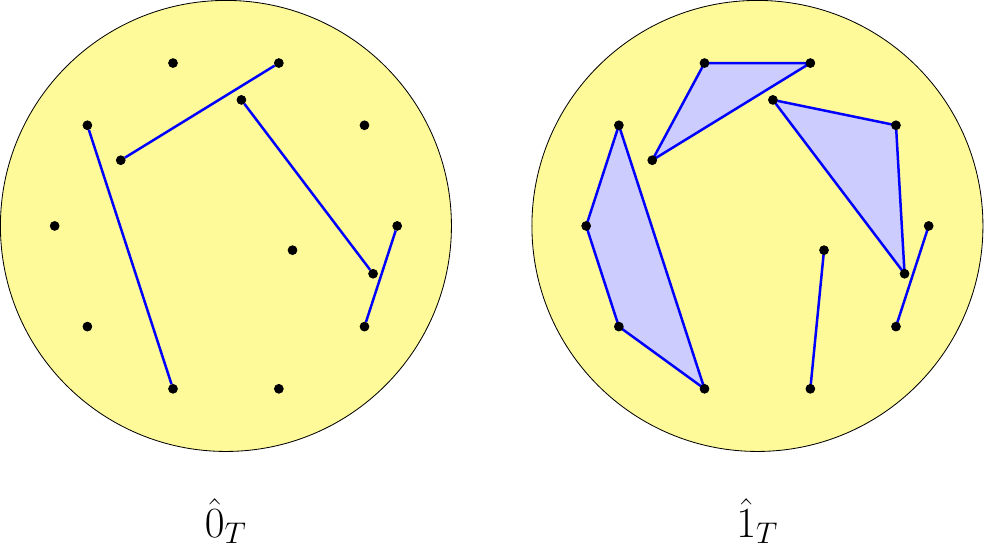}
    \caption{The minimum and maximum elements for the interval associated to the
    skewering tree drawn in Figure~\ref{fig:skewering-tree}}
    \label{fig:skewering-interval}
\end{figure}

\begin{defn}\label{def:skewering-interval}
    Let $P$ be a configuration of $n$ points satisfying Property~$\Delta_2$ and
    let $T$ be a skewering tree for $P$. We define the \emph{skewering interval} 
    $\NC(P,T)$ to be the subposet of $\NC(P)$ consisting of all partitions $\pi$
    satisfying the following conditions:
    \begin{enumerate}
        \item for each $\ell\in T$, the endpoints of $\ell$ share a block in $\pi$;
        \item endpoints of distinct lines in $T$ belong to distinct blocks in $\pi$;
        \item each block in $\pi$ has a convex hull which lies in a cell of $T$.
    \end{enumerate}
    As the name suggests, $\NC(P,T)$ is an interval in $\NC(P)$. 
    Let $\hat{0}_T$ be the partition in $\NC(P)$ for which the only non-singleton blocks
    are $V(\ell)$ for each $\ell\in T$. Let $\hat{1}_T$ be the partition where two points
    in $P$ belong to the same block if and only if they belong to the same cell of $T$
    (note that this requires us to choose a positive side for $\ell_0$ before
    discussing the skewering interval). Then $\NC(P,T)$ is the interval 
    $[\hat{0}_T,\hat{1}_T]$. See Figure~\ref{fig:skewering-interval}.
\end{defn}

We close the section with some features of skewering intervals which will be useful
in the proof of Theorem~\ref{mainthm:rank-symmetry}. To begin, we provide a
product decomposition for skewering intervals.

\begin{defn}\label{def:cell-interval}
    Let $\ell$ be a boundary line in the arrangement $\mathcal{A}$ corresponding to $P$.
    Define $\pi_\ell$ to be the partition of $P$ in which the two endpoints of $\ell$ 
    share a block, while every other block is a singleton. Further, let $\NC_\ell(P)$
    denote the interval $[\pi_\ell,\hat{1}]$ in $\NC(P)$.
\end{defn}

\begin{lem}\label{lem:skewering-intervals-decomposition}
    Let $T$ be a skewering tree for $P$. Then the skewering interval $\NC(P,T)$ is isomorphic
    to the product 
    \[
        \left(\prod_{\substack{\ell\in T\\ \ell\neq \ell_0}} \NC_\ell(V(C_\ell))\right)
        \times
        \NC_{\ell_0}(V(C_{\ell_0}^+))
        \times
        \NC(V(C_{\ell_0}^-)).
    \]
\end{lem}

\begin{proof}
    This follows immediately from Definitions~\ref{def:skewering-interval} and \ref{def:cell-interval}.
\end{proof}

Next, we note that each skewering interval is ``centered'' in the sense that the
rank of its minimum element is equal to the corank of its maximum.

\begin{defn}
    Let $\pi_0$ and $\pi_1$ be elements of $\NC(P)$ with $\pi_0 \leq \pi_1$. 
    We say that the interval $[\pi_0,\pi_1]$ is \emph{centered} if 
    $\rho(\pi_0)+\rho(\pi_1) = |P| - 1$, or equivalently if $bl(\pi_0) + bl(\pi_1) = |P| + 1$.
\end{defn}

\begin{lem}\label{lem:skewering-intervals-centered}
    Each skewering interval $\NC(P,T)$ is a centered subposet of $\NC(P)$.
\end{lem}

\begin{proof}
    Let $T$ be a skewering tree for $P$. Then the partition $\hat{0}_T$ has rank $|T|$ 
    and $\hat{1}_T$ has rank $|P| - (|T|+1)$, so the interval $[\hat{0}_T,\hat{1}_T]$ is
    centered.
\end{proof}

Finally, we examine the ways in which skewering intervals can intersect.

\begin{lem}\label{lem:skewering-intervals-disjoint}
    If $P$ has Property~$\Delta_2$, then distinct maximal skewering 
    trees with the same initial line and choice of positive side 
    yield disjoint skewering intervals.
\end{lem}

\begin{proof}
    Let $T$ and $T'$ be distinct maximal skewering trees and suppose that the partition $\pi$
    is contained in both $\NC(P,T)$ and $\NC(P,T')$. In other words, both
    $\hat{0}_T \leq \pi \leq \hat{1}_T$ and $\hat{0}_{T'} \leq \pi \leq \hat{1}_{T'}$.    
    Since $T$ and $T'$ are distinct and maximal, there must be lines $\alpha,\ell$, and 
    $\ell'$ such that
    $\alpha,\ell \in T$, $\alpha,\ell' \in T'$, and $\alpha$ skewers both $\ell$ and $\ell'$.
    Note that by Lemma~\ref{lem:no-skewer-overlap}, $\ell$ and $\ell'$ cannot skewer one another,
    and by definition of a skewering tree, $\ell$ and $\ell'$ do not intersect internally.
    Therefore the two lines are either parallel or have an external intersection.

    Let $H^-_{\alpha,\ell}$ and $H^-_{\alpha,\ell'}$ denote the closed half-planes bounded by
    $\ell$ and $\ell'$ respectively which do not include the core of $\alpha$. 
    Since $\alpha$ skewers $\ell$, we know that the cell $C_\ell$ (and therefore the core $c(\ell)$)
    must belong to $H^-_{\alpha,\ell}$; the analogous statement holds for $\ell'$.
    By combining the inequalities above, we see that $\hat{0}_T \leq \hat{1}_{T'}$, which means that
    the core $c(\ell)$ must be contained in the cell $C_{\ell'}$, which implies that $c(\ell)$
    belongs to $H^-_{\alpha,\ell'}$. Similarly, the fact that $\hat{0}_{T'} \leq \hat{1}_T$ tells 
    us that $c(\ell')$ is contained in $H^-_{\alpha,\ell}$. 

    Combining all of the above, we know that the intersection 
    $H^-_{\alpha,\ell}\cap H^-_{\alpha,\ell'}$ includes both $c(\ell)$ and $c(\ell')$, but not
    $c(\alpha)$. However, by the same reasoning used in the proof of
    Lemma~\ref{lem:disjoint-cells}, this is precisely the region in which the core of $\alpha$ must 
    be placed for it to skewer both $\ell$ and $\ell'$. Therefore, we have a contradiction, so the 
    skewering intervals 
    for $T$ and $T'$ must be disjoint.
\end{proof}

\section{Property $\Delta_1$ and Catalan Numbers}
\label{sec:delta-1}

Our strategy for proving Theorem~\ref{mainthm:catalan-counting} is
to show that if $P$ satisfies Property~$\Delta_1$, then applying a 
$\Delta_1$-move to $P$ does not change the size of $\NC(P)$. From here,
applying Theorem~\ref{thm:delta-k-connectivity} completes the proof.
We begin with a useful lemma and some terminology, then give the proof
of Theorem~\ref{mainthm:catalan-counting}.

\begin{figure}
    \centering
    \includegraphics[width=0.5\textwidth]{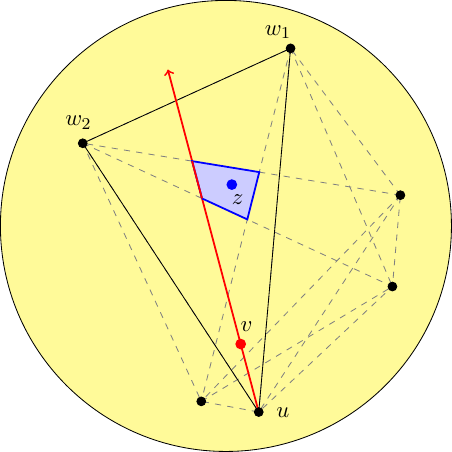}
    \caption{If the region $R_z$ has a side which belongs to a line through an
    interior point $v$, then $z$ and $v$ must both belong to a common triangle.}
    \label{fig:region-convex-polygon}
\end{figure}

\begin{lem}\label{lem:convex-polygon}
    Suppose $P$ satisfies Property $\Delta_1$ and let $z\in \interior(P)$. 
    Then $R_z = R_z^\text{ex}$.
\end{lem}

\begin{proof}
    We know by definition that $R_z$ and $R_z^\text{ex}$ are convex polygons
    with $R_z \subseteq R_z^\text{ex}$, so we just need
    to show that each of its sides is a subset of a line in $\aext$. Suppose 
    that one of the sides for $R_z$ is a subset of a line $\ell$ with endpoints 
    $u$ and $v$, where $v$ (and possibly $u$ as 
    well) is internal. Then $\ell$ must not contain any other points in $P$
    (since $P$ is assumed to be in general position), so it     
    must eventually intersect the boundary of $\conv(P)$ in some edge 
    between extremal vertices $w_1$ and $w_2$, and thus 
    $v$ lies within the triangle with vertex set $\{u,w_1,w_2\}$. 
    Since $\ell$ was assumed to contain a side of $R_z$,
    we know that $z$ must also be contained in the same triangle - see
    Figure~\ref{fig:region-convex-polygon} for an illustration.
    But this implies that $P$ does not satisfy Property~$\Delta_1$, 
    which is a contradiction.
\end{proof}

\begin{rem}\label{rem:adjacent-noncrossing}
    Let $P$ be a configuration satisfying Property $\Delta_1$,
    suppose that $z\in P$ is adjacent to a line $\ell$ in $\aext$, 
    and let $\pi$ be a partition of $P$. 
    Since $P$ satisfies Property~$\Delta_1$, we know by 
    Lemma~\ref{lem:convex-polygon} that $\ell$ contains an edge of
    the region $R_z$. In other words, there are no lines in the
    arrangement $\mathcal{A}^z$ which lie between $\ell$ and $z$.
    This implies that if $B_z$ and $B_\ell$ are the blocks in $\pi$
    containing $z$ and the endpoints of $\ell$ respectively, then 
    $\conv(B_z)$ and $\conv(B_\ell)$ are disjoint if and only if 
    $B_z \cap H_{z,\ell}^- = \emptyset$ and $B_\ell \cap H_{z,\ell}^+ = \emptyset$.
\end{rem}

\begin{defn}\label{defn:pre-post-noncrossing}
    Let $m\colon P \to m(P)$ be a move. We say that $\pi \in \Pi(P)$
    is \emph{pre-$m$-noncrossing} if $\pi$ is noncrossing, but 
    its image $m_*(\pi) \in \Pi(m(P))$ is not. Similarly, we say that 
    $\mu \in \Pi(m(P))$ is \emph{post-$m$-noncrossing} if $\mu$ is 
    noncrossing, but its preimage $m_*^{-1}(\mu) \in \Pi(P)$ is not.
    Then $\NC(m(P))$ can be obtained
    from $\NC(P)$ by the following procedure: remove partitions which are 
    pre-$m$-noncrossing, apply $m_*$ to all remaining elements,
    then add in the post-$m$-noncrossing partitions. 
\end{defn}

We are now ready to prove the main theorem of this section.

\begin{thm}[Theorem~\ref{mainthm:catalan-counting}]
    \label{thm:catalan-counting}
    Let $P \subset \C$ be a configuration of $n$ points
    which satisfies Property~$\Delta_1$. Then $\NC(P)$ is a rank-symmetric graded 
    lattice, and the number of elements with rank $k$ is
    the Narayana number $N_{n,k}$. In particular, $|\NC(P)| = C_n = |\NC_n|$.
\end{thm}

\begin{proof}
    By Theorem~\ref{thm:delta-k-connectivity}, we know there is a sequence of
    $\Delta_1$-moves which transforms $P$ into a convex configuration, for
    which we know the lattice of noncrossing partitions is isomorphic to 
    $\NC_n$. Therefore, we need only show that if $m\colon P\to m(P)$ is a 
    $\Delta_1$-move, then the lattices $\NC(P)$ and $\NC(m(P))$ have the same
    number of elements in each rank.
    
    Our strategy is to show that the block-switching map $\BS_m$
    restricts to a rank-preserving  bijection $\NC(P) \to \NC(m(P))$.
    Suppose that $m\colon P\to m(P)$ is a $\Delta_1$-move which brings 
    the point $z$ across the line $\ell$ in $\aext$, where the 
    endpoints of $\ell$ are $w_1$ and $w_2$, and let $m(z) = y$.

    Let $\pi$ be a partition of $P$. Applying Remark~\ref{rem:adjacent-noncrossing}
    and the fact that $H_{y,\ell}^- = H_{z,\ell}^+$ and $H_{y,\ell}^+ = H_{z,\ell}^-$,
    we have the following sequence of equivalences:
    \begin{align*}
        \text{$\pi$ is pre-$m$-noncrossing} &\leftrightarrow
        \begin{aligned}[t]
            &\text{in $\pi$: $B_z \cap H_{z,\ell}^-$ and 
            $B_\ell \cap H_{z,\ell}^+$ are empty;} \\
            & \text{either $B_{y} \cap H_{y,\ell}^-$ or 
            $B_\ell \cap H_{y,\ell}^+$ is nonempty}
        \end{aligned} \\[1ex]
        &\leftrightarrow
        \begin{aligned}[t]
            &\text{in $\BS_m(\pi)$:  $B_\ell \cap H_{z,\ell}^-$ and 
            $B_{y} \cap H_{z,\ell}^+$ are empty;} \\
            &\text{either $B_{\ell} \cap H_{y,\ell}^-$ or 
            $B_{z} \cap H_{y,\ell}^+$ is nonempty}
        \end{aligned} \\[1ex]
        &\leftrightarrow
        \begin{aligned}[t]
            &\text{in $\BS_m(\pi)$: $B_\ell \cap H_{y,\ell}^+$ and 
            $B_{y} \cap H_{y,\ell}^-$ are empty;} \\
            &\text{either $B_{\ell} \cap H_{z,\ell}^+ $ or 
            $B_{z} \cap H_{z,\ell}^-$ is nonempty}
        \end{aligned} \\[1ex]
        &\leftrightarrow
        \text{$\BS_m(\pi)$ is post-$m$-noncrossing}
    \end{align*}
    Thus, the block-switching map $\BS_m$ induces a rank-preserving bijection 
    between the pre-$m$-noncrossing partitions of $P$ and the post-$m$-noncrossing 
    partitions of $m(P)$, and we are done.
\end{proof}

\section{Property $\Delta_2$ and Rank Symmetry}
\label{sec:delta-2}

\begin{figure}
	\centering
    \includegraphics[width=0.8\textwidth]{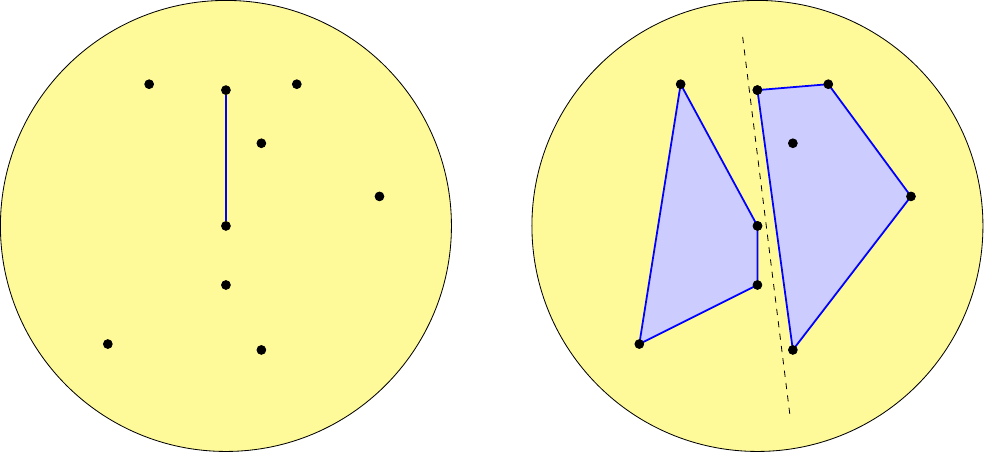}
    \caption{An atom in $\NC(P)$ with the corresponding coatom}
    \label{fig:ncp-atomCoatom}
\end{figure}

We now turn our attention to Theorem~\ref{mainthm:rank-symmetry},
in which we demonstrate that if $P$ satisfies Property~$\Delta_2$, 
then $\NC(P)$ is rank-symmetric. As a first step, one could write an
explicit bijection between the atoms and coatoms of $\NC(P)$,
where $P$ is an arbitrary configuration of $n$ points. This was
previously demonstrated by Razen and Welzl in the special case
where $P$ is in general position \cite{razen13} and is straightforward
to generalize. In short: each atom is determined by a single line
in $\mathcal{A}$, and by slightly rotating this line counterclockwise 
about the midpoint of its core, we obtain a line which divides $P$ into
two pieces, thus producing a coatom in $\NC(P)$ - see Figure
\ref{fig:ncp-atomCoatom} for an illustration.

One could hypothetically prove Theorem~\ref{mainthm:rank-symmetry} 
by extending the map above to a rank-reversing bijection from $\NC(P)$
to itself, but this seems intractable in general. Instead, our proof
technique is similar to that of Theorem~\ref{mainthm:catalan-counting}, although
the weakening of our hypotheses from Property~$\Delta_1$ to 
Property~$\Delta_2$ means that Lemma~\ref{lem:convex-polygon} no longer holds.
That is, it is possible that while some points in $P$ are adjacent to 
non-separating lines in $\mathcal{A}$, none of these lines are in $\aext$.
As a result, we must account for moves which take an interior point across a
line with an interior endpoint, which in turn means we need to understand situations
where a partition block overlaps with the line being moved across. 
To do so, we use the results on skewers developed in Section~\ref{sec:skewers}.

First, we
require a technical lemma regarding the interval $\NC_\ell(P)$ introduced in
Definition~\ref{def:cell-interval}. For the remainder of this section, let $P$ 
denote a configuration of $n$ points in $\C$ which satisfies Property~$\Delta_2$.

\begin{defn}
    Let $\ell$ be a boundary line in the arrangement $\mathcal{A}$ corresponding to $P$.
    We say that $P$ \emph{satisfies Property~$\Delta_1$ relative to $\ell$} if each
    $B\subseteq P$ with $V(\ell) \subseteq B$ has at most 
    $\lfloor \frac{|B|-2}{2} \rfloor$ internal points.
\end{defn}

The motivation for the preceding definition comes from its appearance in
certain skewering trees for configurations which satisfy Property~$\Delta_2$.

\begin{lem}\label{lem:skewering-trees-relative-delta-1}
    Let $T$ be a skewering tree for $P$ with initial line $\ell_0$. 
    Suppose that $\ell_1$ is a line in $T$ with $\ell_0 \dashv \ell_1$,
    and that $y$ is a point in $P$ which lies in the convex hull $c(\ell_0,\ell_1)$ 
    on the non-positive side of $\ell_0$. Then 
    $V(C_{\ell_0}^+)$ satisfies Property~$\Delta_1$ relative to $\ell_0$,
    and for each $\ell\in T$ with $\ell \neq \ell_0$, 
    $V(C_\ell)$ satisfies Property~$\Delta_1$ relative to $\ell$.
\end{lem}

\begin{proof}
    We prove both claims by contradiction. To start, suppose that $V(C_{\ell_0}^+)$
    does not satisfy Property~$\Delta_1$ relative to $\ell_0$. Then there is a 
    subset $B\subseteq V(C_{\ell_0}^+)$ which contains the endpoints of $\ell_0$
    such that $B$ has more than $\lfloor \frac{|B|-2}{2} \rfloor$ internal points.
    If we define $B' = B\cup V(\ell_1) \cup \{y\}$, then the internal points of $B$,
    together with $y$ and one endpoint of $\ell_0$, are all internal points of $B'$.
    Therefore, $B'$ has more than
    \[
        \left\lfloor \frac{|B| - 2}{2}\right\rfloor + 2 
        = \left\lfloor \frac{|B| + 2}{2}\right\rfloor
        = \left\lfloor \frac{|B'| - 1}{2} \right\rfloor
    \]
    internal points, which violates the assumption that $P$ satisfies Property~$\Delta_2$.
    Thus $V(C_{\ell_0}^+)$ satisfies Property~$\Delta_1$ relative to $\ell_0$.

    Similarly, let $\ell\in T$ with $\ell\neq \ell_0$ and suppose that there is a 
    subset $B\subseteq V(C_\ell)$ which contains the endpoints of $\ell$ such that
    $B$ has more than $\lfloor \frac{|B|-2}{2} \rfloor$ internal points. 
    If $\ell = \ell_1$, then we can define $B' = B\cup V(\ell_0) \cup \{y\}$ 
    and observe by the same reasoning as above that $B'$ has more than
    $\lfloor \frac{|B'|-1}{2} \rfloor$ internal points, which provokes a contradiction.
    Suppose instead that $\ell \neq \ell_1$. Then there is a skewering sequence 
    $\alpha_1 \dashv \cdots \dashv \alpha_k \dashv \ell$ such that
    either $\ell_0 \dashv \alpha_1$ or $\ell_0 \dashv \ell_1 \dashv \alpha_1$.
    In either case, define
    \[
        B' = B\cup V(\ell_0) \cup V(\ell_1) \cup \{y\} \cup V(\alpha_1) \cup \cdots \cup V(\alpha_k)
    \]
    and observe that the internal points of $B'$ include all the internal 
    points of $B$, as well as one endpoint of each $\alpha_i$, one endpoint of $\ell_0$,
    and $y$. Therefore, the number of internal points in $B'$ is more than
    \[
        \left\lfloor \frac{|B|-2}{2} \right\rfloor + k + 3 
        = \left\lfloor \frac{|B|+2k+4}{2} \right\rfloor
        = \left\lfloor \frac{|B'|-1}{2} \right\rfloor,
    \]
    which violates Property~$\Delta_2$. 
    Thus, $V(C_\ell)$ must have Property~$\Delta_1$ relative to $\ell$.
\end{proof}

\begin{lem}\label{lem:interval-lemma}
    Suppose that $P$ satisfies Property~$\Delta_1$ relative to the boundary line 
    $\ell$ and that for each proper subset $Q\subset P$, the poset of noncrossing partitions 
    $\NC(Q)$ is rank-symmetric. Then $\NC_\ell(P)$ is rank-symmetric.
\end{lem}

\begin{proof}
    We proceed by induction on the number of internal points of $P$. 
    First, if $P$ has no internal points, then since $\ell$ was assumed to be
    a boundary line, we can see that $\NC_\ell(P)$ is
    isomorphic to $\NC_{n-1}$, which is rank-symmetric. Now, suppose that the
    claim is true for any configuration with up to $k-1$ internal points which
    satisfies the lemma's hypotheses, and let $P$ be a configuration with $k$ 
    internal points such that for all proper subsets $Q\subset P$, we know that
    $\NC(Q)$ is rank-symmetric. Let $u$ and $v$ be the endpoints of $\ell$,
    and let $P^v$ denote the complement $P - \{v\}$. We will compare the interval 
    $[\pi_\ell,\hat{1}]\subset \NC(P)$ to the noncrossing partition lattice
    $\NC(P^v)$, which we know is rank-symmetric by assumption.

    Define the map $\phi_v \colon \NC_\ell(P) \to \NC(P^v)$
    by removing $v$ from each partition in the domain and observe that $\phi_v$
    is always injective, but typically not surjective. Our goal is to
    show that $\NC_\ell(P)$ is rank-symmetric; since $\NC(P^v)$
    is assumed to be rank-symmetric and $\phi_v$ is injective, it suffices to 
    show that the complement $\NC(P^v) - \phi_v(\NC_\ell(P))$ is a union of 
    centered, disjoint, rank-symmetric intervals.
    
    Let $W\subset P -\{u,v\}$ be the set of all points $w$ with the property that 
    $\interior(\{u,v,w\})$ is nonempty and note that since $P$ satisfies 
    Property~$\Delta_1$ relative to $\ell$, we must have $|\interior(\{u,v,w\})| = 1$. 
    Then $\NC(P^v) - \phi_v(\NC_\ell(P))$ is the collection of all partitions 
    $\sigma$ of $P^v$ with a block which contains both $u$ and a point $w\in W$, 
    but not the unique point in $\interior(\{u,v,w\})$. Rephrasing this
    characterization in the language of skewering trees, let
    $\mathcal{T}$ be the collection of skewering trees for $P^v$ such that 
    the initial line $\ell_0$ has endpoints $u$ and $w$ for some $w\in W$ and 
    the positive side of $\ell_0$ is chosen to be the one which does not 
    include $v$; then a partition $\sigma \in \NC(P^v)$ lies in 
    $\NC(P^v) - \phi_v(\NC_\ell(P))$ if and only if it belongs to the 
    skewering interval for some skewering tree in $\mathcal{T}$. Together with 
    Lemmas~\ref{lem:skewering-intervals-centered} and
    \ref{lem:skewering-intervals-disjoint}, this tells us that the
    skewering intervals for trees in $\mathcal{T}$ form a collection of centered and disjoint
    intervals whose union is $\NC(P^v) - \phi_v(\NC_\ell(P))$. All that remains is to show
    that each such skewering interval is rank-symmetric.

    Fix a skewering tree $T \in \mathcal{T}$ and consider the skewering interval
    $\NC(P^v,T)$. By Lemma~\ref{lem:skewering-intervals-decomposition}, 
    we have a poset isomorphism 
    \[  
        \NC(P^v,T) \cong 
        \left(\prod_{\substack{\alpha\in T \\ \alpha\neq \ell_0}} \NC_\alpha(P^v \cap C_\alpha) 
        \right) \times
        \NC_{\ell_0}(P^v \cap C^+_{\ell_0})\times
        \NC(P^v \cap C^-_{\ell_0}).
    \]
    By our assumption that every proper subset of $P$ has a rank-symmetric lattice of noncrossing
    partitions, we know that $\NC(P^v\cap C^-_{\ell_0})$ is rank-symmetric. For the middle term
    in the product, we claim that the set $P^v \cap C^+_{\ell_0}$ satisfies Property~$\Delta_1$
    relative to the boundary line $\ell_0$. To see this, 
    recall that the endpoints of $\ell_0$ are the extremal point $u$ and some $w\in W$, and
    define $w'$ be the unique point of $P$ which lies in the interior of $\conv(\{u,v,w\}$.
    If there were points $x, x', x'' \in P^v \cap C^+_{\ell_0}$ such that
    both $x'$ and $x''$ lie in the convex hull $\conv(\{u,w,x\})$, then $\{u,v,w,w',x,x',x''\}$ 
    would be a set of seven points, at least three of which are internal ($w'$, $x'$, and $x''$).
    This would violate our assumption that $P$ has Property~$\Delta_1$ relative to $\ell$, so
    it must be the case that $P^v - C^+_{\ell_0}$ satisfies Property~$\Delta_1$
    relative to $\ell_0$, and by our inductive hypothesis, we know that 
    $\NC_{\ell_0}(P^v \cap C^+_{\ell_0})$ is rank-symmetric.

    Generalizing this argument, we now show that $P^v \cap C_\alpha$ satisfies Property~$\Delta_1$ 
    relative to $\alpha$ for each non-initial line $\alpha$ in $T$. 
    Let $\alpha_1 \dashv \cdots \dashv \alpha_m$ be a skewering sequence in $T$ such that 
    $\alpha_1 = \ell_0$ and $\alpha_m = \alpha$. If there are points
    $x,x',x'' \in P^v \cap C_\alpha$ such that $x'$ and $x''$ lie in the triangle formed by $x$ 
    and the endpoints of $\alpha$, and if $w$ and $w'$ are defined as above, then 
    \[
        \{x,x',x'',w',v\} \cup V(\alpha_1) \cup \cdots \cup V(\alpha_m)
    \]
    is a set of $2m+5$ points, of which at least $m+2$ points must be internal - see 
    Figure~\ref{fig:interval-lemma} for an illustration. Since this set contains both
    $u$ and $v$, this violates our assumption that $P$ satisfies Property~$\Delta_1$ relative to
    $\ell$, so we may conclude that $P^v \cap C_\alpha$ satisfies Property~$\Delta_1$ 
    relative to $\alpha$, as desired. By the inductive hypothesis, $\NC_\alpha(P^v \cap C_\alpha)$
    is rank-symmetric.

    \begin{figure}
        \centering
        \includegraphics[width=\textwidth]{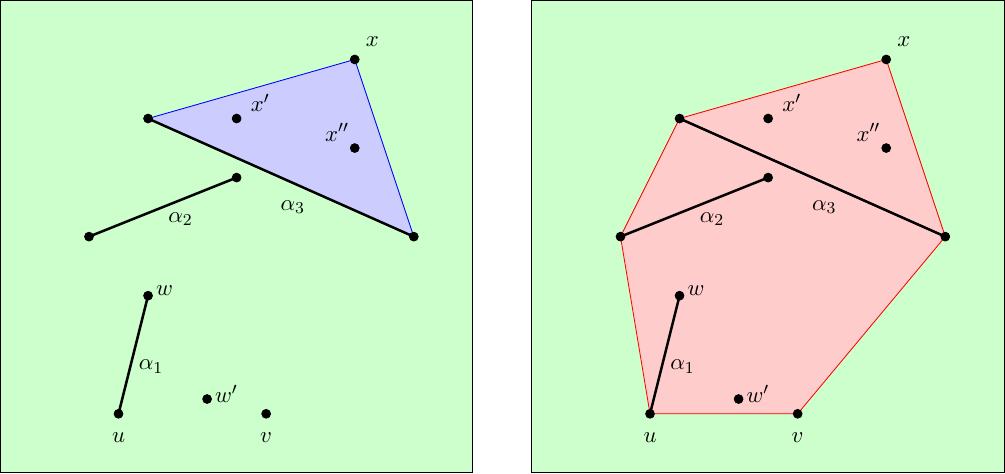}
        \caption{If $\alpha_1\dashv\alpha_2\dashv\alpha_3$ is a skewering sequence
        for $P$ and the points in the cell for $\alpha_3$ do not satisfy Property~$\Delta_1$
        relative to $\alpha_3$ (illustrated on the left in blue), then there is
        a set of points containing $u$ and $v$ (illustrated on right in red) which
        demonstrate that $P$ does not satisfy Property~$\Delta_1$ relative to the
        line with endpoints $u$ and $v$.}
        \label{fig:interval-lemma}
    \end{figure}
    
    Finally, we have that $\NC(P^v,T)$ is a product of rank-symmetric posets and is therefore
    rank-symmetric itself, which completes the proof.
\end{proof}

\begin{thm}[Theorem~\ref{mainthm:rank-symmetry}]
    \label{thm:rank-symmetry}
    Let $P \subset \C$ be a set of $n$ distinct points which satisfies 
    Property~$\Delta_2$. Then $\NC(P)$ is a rank-symmetric graded lattice.
\end{thm}

\begin{proof}
    We proceed by induction on the number of internal points for $P$. 
    When $P$ has no internal points, $\NC(P)$ is isomorphic to the classical
    noncrossing partition lattice $\NC_n$, which is rank-symmetric. Now, 
    suppose that every configuration with fewer than $|\interior(P)|$ internal points 
    has a rank-symmetric lattice of noncrossing partitions; we will show that 
    $\NC(P)$ is rank-symmetric as well. By Theorem~\ref{thm:delta-k-connectivity}, 
    there is a sequence of $\Delta_2$-moves which transforms $P$ into a 
    convex configuration, for which the lattice of noncrossing partitions 
    is isomorphic to $\NC_n$ and thus 
    rank-symmetric. The only remaining step is to prove that if
    $m\colon P \to m(P)$ is a $\Delta_2$-move, then $\NC(P)$ is rank-symmetric
    if and only if $\NC(m(P))$ is rank-symmetric. 

    In the proof of Theorem~\ref{mainthm:catalan-counting}, we showed that when $P$
    satisfies Property~$\Delta_1$, the block-switching map $\BS_m$ is a 
    rank-preserving bijection between pre-$m$-noncrossing partitions of $P$ and 
    post-$m$-noncrossing partitions of $m(P)$. If $P$ is only assumed to satisfy
    Property~$\Delta_2$, the block switching map might not be a bijection; there may
    be partitions $\pi\in \NC(P)$ such that $\BS_m(\pi)$
    is not in $\NC(m(P))$, or elements $\sigma\in \NC(m(P))$ such that $\BS_m^{-1}(\sigma)$ 
    does not lie in $\NC(P)$. To complete the proof, we must show that these two 
    collections are rank-symmetric subposets of $\NC(P)$ and $\NC(m(P))$ respectively. 
    By symmetry, it suffices to examine the first collection.

    Suppose that the move $m$ takes a point $z\in P$ across a line $\ell \in \mathcal{A}$
    with endpoints $e(\ell) = \{w_1,w_2\}$ and let $F_m(P)$ denote the noncrossing 
    partitions of $P$ which fail to be accounted
    for by the block-switching map $\BS_m$; that is, 
    \[
        F_m(P) = \{\pi \in \NC(P) \mid \BS_m(\pi) \not\in \NC(m(P))\}.
    \]
    If $F_m(P)$ is empty, then there is nothing to prove. Otherwise, for each
    $\pi \in F_m(P)$, there are points $x,y\in P$ such that 
    $y,z \in \interior(\{w_1,w_2,x\})$, and these five points are divided into 
    three distinct blocks of $\pi$ (each of which might contain other points) as 
    follows: $x$ and $z$ belong to one block, $w_1$ and $w_2$ belong to another, 
    and $y$ lies in a third - see Figure~\ref{fig:block-switch-fail}. 
    Using this characterization, we will prove that 
    (1) $F_m(P)$ is a union of skewering intervals, (2) each of these skewering intervals 
    is centered and rank-symmetric, and (3) intersections of the skewering intervals 
    are centered and rank-symmetric.

    \begin{figure}
        \centering
        \includegraphics[width=0.5\textwidth]{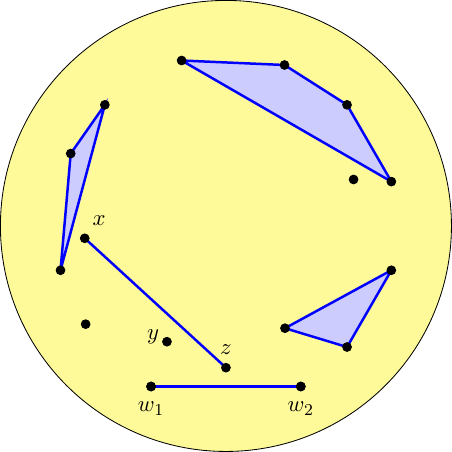}
        \caption{For the depicted partition $\pi$ of $P$, if $m$ moves $z$
        across the line containing $w_1$ and $w_2$, then $\BS_m(\pi)$ is not an
        element of $\NC(m(P))$ since it has $w_1$, $w_2$, and $x$ together
        in a single block without the interior point $y$.}
        \label{fig:block-switch-fail}
    \end{figure}
    
    More concretely, let $y_1,\ldots,y_k$ be the points in $P$ such that for each 
    $i\in \{1,\ldots,k\}$, there is some $x_i\in P$ such that $w_1$, $w_2$, and 
    $x_i$ form a triangle with interior points $z$ and $y_i$ (and nothing else, since
    $P$ satisfies Property~$\Delta_2$). When there is more than one choice for $x_i$, 
    we select the unique option where all other possible
    choices lie in the half-plane bounded by $x_i$ and $z$ which does not contain $y_i$.
    Then $F_m(P)$ consists of all partitions in $\NC(P)$ where there exists an $i$
    such that $z$ and $x_i$ belong to the same block, $w_1$ and $w_2$ belong to 
    a different block, and $y_i$ belongs to a third block. 

    Let $\ell_i$ denote the line with endpoints $x_i$ and $z$. Since $z$ is adjacent 
    to the line $\ell$, we know that $\ell_i$ must skewer $\ell$. The 
    characterization above can thus be rephrased: the partition $\pi$ belongs to $F_m(P)$
    if and only if $\pi$ lies in a skewering interval $\NC(P,T_i)$ for some 
    skewering tree $T_i$ with initial line $\ell_i$
    such that $\ell \in T$ and $y_i$ is on the non-positive side of $\ell_i$.
    It is straightforward to see that each such skewering interval is a subset
    of $F_m(P)$, so we may conclude that $F_m(P)$ is a union of skewering intervals.

    Each skewering interval in the union is centered by 
    Lemma~\ref{lem:skewering-intervals-centered}, which implies that 
    $F_m(P)$ is itself a centered subposet of $\NC(P)$. 
    By Lemma~\ref{lem:skewering-intervals-decomposition}, each skewering interval
    $\NC(P,T_i)$ decomposes into a product of posets; the factor $\NC(V(C_{\ell_i}^-))$
    is rank-symmetric by our inductive hypothesis and terms of the form
    $\NC_\ell(V(C_\ell))$ and $\NC_{\ell_i}(V(C_{\ell_i}^+))$ are rank-symmetric
    by Lemmas~\ref{lem:skewering-trees-relative-delta-1} and \ref{lem:interval-lemma}.
    Since the product of rank-symmetric posets is itself rank-symmetric, we may conclude 
    that the interval $\NC(P,T_i)$ is rank-symmetric.

    \begin{figure}
        \centering
        \includegraphics[width=\textwidth]{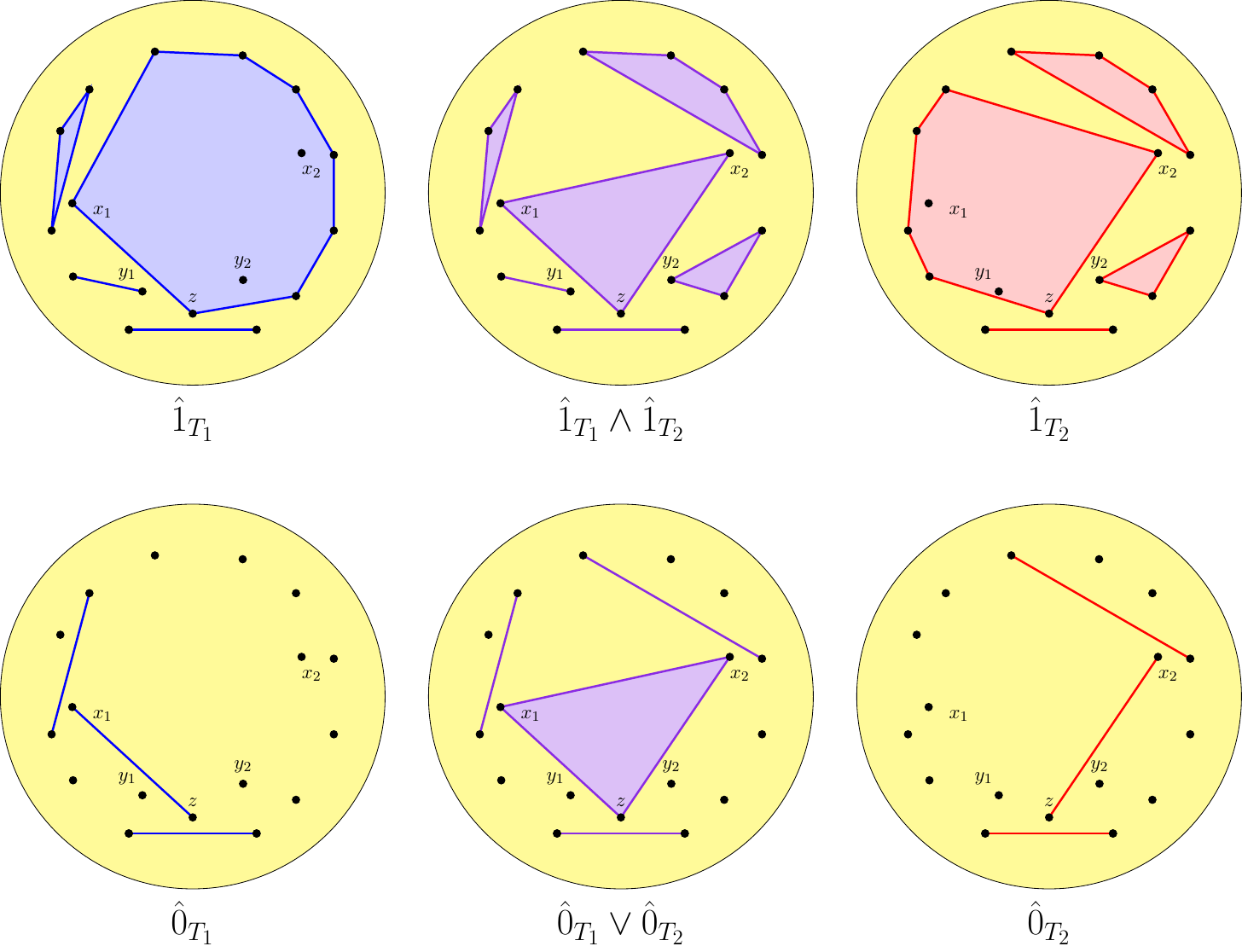}
        \caption{On the left and right, the minimum and maximum elements for 
        a skewering interval are given. The intersection of these two intervals
        is another interval, albeit one which does not arise from a skewering tree;
        its minimum and maximum elements are displayed in the center.}
        \label{fig:interval-intersection}
    \end{figure}

    We are now at the final step: examining how the skewering intervals which
    make up $F_m(P)$ can intersect. First, note that if $T_i$ and $T_i'$ are 
    two distinct skewering trees in the union which both use $\ell_i$ as the initial
    line, we know by Lemma~\ref{lem:skewering-intervals-disjoint} that 
    $\NC(P,T_i)\cap\NC(P,T_i')$
    is empty. Next, consider two skewering trees $T_i$ and $T_j$ which 
    appear in the union with $i\neq j$. The intersection $\NC(P,T_i)\cap\NC(P,T_j)$
    is nonempty precisely when the cells $C_{\ell_i}^+$ for $T_i$ and $C_{\ell_j}^+$
    for $T_j$ have an intersection which excludes both $y_i$ and $y_j$
    (note that this precludes the possibility of nonempty triple intersections
    for skewering intervals). When this is
    the case, the two skewering intervals intersect in the interval
    $[\hat{0}_{T_i} \vee \hat{0}_{T_j}, \hat{1}_{T_i} \wedge \hat{1}_{T_j}]$;
    this is not a skewering interval, but it has many of the associated properties.
    By similar arguments to those given in Lemma~\ref{lem:no-skewer-overlap}, 
    Lemma~\ref{lem:disjoint-cells} and Lemma~\ref{lem:cells-cover}, we can see that
    the minimum element $\hat{0}_{T_i} \vee \hat{0}_{T_j}$ consists of an edge for
    each non-initial line in $T_i$ and $T_j$, together with the triangle with 
    vertices $y_i$, $y_j$, and $z$. This triangle has three cells associated to it:
    one containing $z_i$, one containing $z_j$, and one which is the intersection 
    of $C_{\ell_i}^+$ and $C_{\ell_j}^+$ (thus containing the triangle itself). 
    Each other edge has a well-defined cell in the same way as a typical 
    skewering tree does. The maximal element $\hat{1}_{T_i} \wedge \hat{1}_{T_j}$ 
    is constructed using these cells in the same manner as for skewering trees -
    see Figure~\ref{fig:interval-intersection} for an illustration.
    Putting this all together, 
    $[\hat{0}_{T_i} \vee \hat{0}_{T_j}, \hat{1}_{T_i} \wedge \hat{1}_{T_j}]$
    admits a decomposition similar to the one described in 
    Lemma~\ref{lem:skewering-intervals-decomposition}, so by 
    Lemmas~\ref{lem:skewering-trees-relative-delta-1} and \ref{lem:interval-lemma}, 
    this interval is centered and rank-symmetric.

    In summary, we have shown that $F_m(P)$ is a union of centered and rank-symmetric 
    intervals in $\NC(P)$, whose pairwise intersections are themselves centered 
    and rank-symmetric and whose $k$-wise intersections are empty when $k>2$. Therefore,
    $F_m(P)$ is centered and rank-symmetric, which completes the proof.
\end{proof}

\section*{Acknowledgments}
We are grateful to Swarthmore College for support during the summer of 2022, 
when the majority of our research took place. We also thank the referees for their
helpful comments, especially in clarifying Propositions~\ref{prop:ncp-lattice}
and \ref{prop:graded}.

\section*{Data Availability Statement}

This manuscript has no associated data.

\end{document}